\documentclass[10pt]{amsart}

\newtheorem{thm}{Theorem}[section]
\newtheorem{prop}[thm]{Proposition}
\newtheorem{conj}[thm]{Conjecture}

\newtheorem{lem}[thm]{Lemma}

\theoremstyle{definition}

\numberwithin{equation}{section}
\newtheorem{rem}[thm]{Remark}
\newtheorem{ex}[thm]{Example}
\newtheorem{defn}[thm]{Definition}

\begin{document}
\bibliographystyle{amsalpha}

\title[Rogers dilogarithms of higher degree]
{
Rogers dilogarithms of higher degree and generalized cluster algebras
}

\author{Tomoki Nakanishi}
\address{\noindent Graduate School of Mathematics, Nagoya University, 
Chikusa-ku, Nagoya,
464-8604, Japan}
\email{nakanisi@math.nagoya-u.ac.jp}

\thanks{This work was partially supported by JSPS KAKENHI Grant Number 16H03922.
}

\date{}
\maketitle
\begin{abstract}
In connection with generalized cluster algebras
we introduce a certain generalization of
the celebrated Rogers dilogarithm,
which we call the Rogers dilogarithms of higher degree.
We show that there is an identity of these generalized Rogers dilogarithms 
associated with any period of seeds of a
generalized cluster algebra.
\end{abstract}

\section{Introduction}
It is widely known that the {\em (Euler) dilogarithm}
\begin{align}
\label{eq:Li2ps2}
\mathrm{Li}_2(x)=\sum_{n=1}^{\infty} \frac{x^n}{n^2}
=-\int_0^x \frac{\log(1-y)}{y}
dy
\end{align}
appears and plays important roles in several branches of mathematics
(e.g., \cite{Lewin81,Kirillov95,Zagier07}).
The function is remarkable in the sense that it satisfies
a variety of functional equations,
which are generally called {\em dilogarithm identities}.

The {\em quantum dilogarithm} \cite{Faddeev93,Faddeev94}
\begin{align}
\label{eq:qd1}
\mathbf{\Psi}_q(x)
=
\prod_{k=0}^{\infty}
(1+q^{2k+1}x)^{-1},
\end{align}
is regarded as a quantum analogue of the dilogarithm,
and it is as import as the classical (i.e., nonquantum) one \eqref{eq:Li2ps2}.
It is related to its classical counterpart
in  the asymptotic limit as follows:
\begin{align}
\label{eq:asym1}
\mathbf{\Psi}_q(x)
\sim
\exp\left(-
\frac{\mathrm{Li}_2(-x)}{\log q^2}
\right),
\quad
q\rightarrow 1^{-}.
\end{align}

The classical and quantum dilogarithms are intimately related to
the {\em cluster algebras} \cite{Fomin02,Fomin07} and {\em quantum cluster algebras} \cite{Fock03,Fock07}, 
respectively.
This connection was  discovered via the quantization of the  moduli space
of  Riemann surfaces \cite{Fock99,Fock03,Fock07}.
Then, it was also spotlighted through the Donaldson-Thomas theory
\cite{Kontsevich08,Nagao11,Keller11},
and through the $Y$-systems in conformal field theory \cite{Chapoton05, Nakanishi09,
Nakanishi10c}.
In particular, we  reached to the following very general theorem.
\begin{thm}
\label{thm:intro1}
\par
\begin{itemize}
\item[(I).]
There is a dilogarithm identity associated with any period of seeds
of a cluster algebra \cite{Nakanishi10c}.
\item[(II).]
There is a quantum dilogarithm identity associated with any period of seeds
of a quantum cluster algebra \cite{Keller11,Kashaev11}.
\end{itemize}
\end{thm}

A heuristic (but not rigorous) derivation of
Statements (I) from Statement (II) was also given in \cite{Kashaev11}.

Recently Chekhov and Shapiro introduced {\em generalized cluster algebras}
\cite{Chekhov11}.
As the name suggests it is a generalization of cluster algebras.
In fact, it is a {\em very natural} generalization so that 
all nice properties of cluster algebras are shown or expected to hold
\cite{Rupel13,Nakanishi14a}.
Naturally one can further define the {\em quantum generalized cluster algebras}
as well \cite{Nakanishi14b}.
For   ordinary quantum cluster algebras, the quantum dilogarithm \eqref{eq:qd1}
plays a key role to control their mutations \cite{Fock03,Fock07}.
 Likewise, 
for  generalized quantum cluster algebras, a generalization of
the quantum dilogarithm, called the {\em quantum dilogarithms of higher degree}
 play the same role,
 and they look as follows: (This notation is used   only here.)
\begin{align}
\label{eq:gqd1}
\mathbf{\Psi}_q^P(x)
=
\prod_{k=0}^{\infty}
P(q^{2k+1}x)^{-1},
\end{align}
where $P(x)$ is an arbitrary monic polynomial in $x$ with unit constant;
furthermore, here we assume that all coefficients of $P(x)$ are
nonnegative real numbers.

In the simplest case $\deg P(x)=1$, we have $P(x)=1+x$ and it reduces to 
the ordinary one \eqref{eq:qd1}.
(In \cite{Nakanishi14b} $P(x)$ is assumed to be {\em reciprocal}, but
this assumption can be  removed with slight change of mutations therein. See \cite{Nakanishi15}.)
Then, it is rather straightforward to generalize Statement (II) in Theorem
\ref{thm:intro1}
in the following way:
\begin{thm}
\begin{itemize}
\item[(II').] There is a quantum dilogarithm identity of higher degree associated with any period of seeds of a quantum generalized cluster algebra \cite{Nakanishi14b}.
\end{itemize}
\end{thm}

Under this circumstance
 it is just natural to expect the following generalization of Statement (I),
 which is also the classical counterpart of Statement (II'):
\begin{itemize}
\item[(I').] {\em There is a dilogarithm identity of higher degree associated with any period of seeds
of a generalized cluster algebra.}
\end{itemize}
In fact, a classical counterpart of the function \eqref{eq:gqd1} was already introduced
in \cite{Nakanishi14b},
and it looks as follows:  (This notation is used   only here.)
\begin{align}
\label{eq:gLi2ps2}
\mathrm{Li}_2^P(x)
=-\int_0^{-x} \frac{\log P(y)}{y}
dy.
\end{align}
Two functions in \eqref{eq:gqd1} and \eqref{eq:gLi2ps2} are related 
in 
 the same way as \eqref{eq:asym1}:
\begin{align}
\label{eq:asym2}
\mathbf{\Psi}_q^P(x)
\sim
\exp\left(
-
\frac{\mathrm{Li}_2^P(-x)}{\log q^2}
\right),
\quad
q\rightarrow 1^{-}.
\end{align}

Now let us explain the obstacle to establish Statement (I') we have been confronted with so far.
In the ordinary case the derivation of Statement (I) in  \cite{Nakanishi10c} is
not as straightforward as its quantum counterpart (II).
The reason is that the dilogarithm naturally concerning with the identities in Statement (I)
is not exactly the Euler dilogarithm \eqref{eq:Li2ps2};
rather, it is the {\em Rogers dilogarithm} defined by
\begin{align}
\label{eq:Lint01}
L(x)
&=
-\frac{1}{2}\int_0^{x} 
\left\{
\frac{\log(1-y)}{y}
+
\frac{\log y}{1-y}
\right\} dy,
\end{align}
where two dilogarithms are related by
\begin{align}
L(x)&= \mathrm{Li}_2(x) + \frac{1}{2}\log x \cdot \log(1-x).
\end{align}
Then, the main issue is to set the proper definition
of the Rogers dilogarithms of higher degree so as to match 
the desired Statement (I'),
which is not {\em a priori} clear.

In this paper we resolve this problem,
and we give a proper definition of the {\em Rogers dilogarithms of higher degree}.
The key idea is to generalize the Rogers dilogarithm,
not through
 the familiar definition \eqref{eq:Lint01},
but through a  less familiar expression,
\begin{align}
\label{eq:Lhd0}
L\left(\frac{x}{1+x}\right)
&=
\frac{1}{2}\int_0^{x} 
\left\{
\frac{\log(1+y)}{y}
-
\frac{\log y}{1+y}
\right\} dy,
\quad
(0 \leq x).
\end{align}
Then, for the same polynomial $P(x)$ in 
\eqref{eq:gLi2ps2}, we define the corresponding Rogers dilogarithm of higher degree
as follows:
(This notation is used   only here.)
\begin{align}
\label{eq:Lhd1}
L^P\left(\frac{x^{\mathrm{deg} P}}{P(x)}\right)
&=
\frac{1}{2}\int_0^{x} 
\left\{
\frac{\log P(y)}{y}
-
\frac{\log y}{P(y)}P'(y)
\right\} dy,
\quad
(0 \leq x).
\end{align}

Once this part is cleared,
it is rather straightforward to follow the argument of the proof of
Statement (I),
and we prove Statement (I'),
which is our main result (Theorems \ref{thm:gid2} and \ref{thm:gid3}).
To complete the picture, we also give a heuristic derivation
of Statement (I') from Statement (II') following \cite{Kashaev11}.

The function  \eqref{eq:Lhd1} may formally reduce
to  (the analytic continuation of) the ordinary one \eqref{eq:Lhd0}
if we factorize the polynomial $P(x)$  into   polynomials of degree one
with complex coefficients.
However, we need to correctly choose
 the path of  the analytic continuation 
of \eqref{eq:Lhd0}
very carefully, and that makes the things  very complicated.
Therefore, it is natural to keep the function 
\eqref{eq:Lhd1} as a single package.

The organization of the paper is as follows.
In Section 2 the connection between
the Rogers dilogarithm and cluster algebras is reviewed.
In Section 3 we introduce the Rogers dilogarithms of higher degree
and prove the constancy condition.
In Section 4 we show that there is a dilogarithm identity of higher degree
associated with any period of a generalized cluster algebra.
In Section 5 we give a proof of Theorem \ref{thm:gyconst1},
which is the key theorem for our main result.
In Section 6 some explicit examples are given.
In Appendix, as  independent reading, we give a heuristic (but not rigorous) derivation of  our 
dilogarithm identity of higher degree from the quantum one.

\medskip
\noindent
{\em Acknowledements.}
The author would like to thank the referee for  useful comments
and suggestions.

\section{Rogers dilogarithm and cluster algebras}
In this section we quickly recall and summarize some known properties of
the Rogers dilogarithm and its relation to cluster algebras,
which we will generalize in this paper.
The whole section may serve as a useful guide 
to show where we are heading.
Here we do not provide any proof,
but the  main theorems will be reproved 
in a more general setting   in the later sections.

\subsection{Rogers dilogarithm}
The  {\em Euler dilogarithm\/} $\mathrm{Li}_2(x)$
was originally defined by the following power series
with convergence radius 1 (see \cite{Lewin81,Kirillov95,Zagier07} for backgrounds):
\begin{align}
\label{eq:Li2ps1}
\mathrm{Li}_2(x)=\sum_{n=1}^{\infty} \frac{x^n}{n^2}.
\end{align}
We have
\begin{align}
\label{eq:Li2const1}
\mathrm{Li}_2(0)=0,
\quad
\mathrm{Li}_2(1)=\zeta(2)=\frac{\pi^2}{6},
\end{align}
where $\zeta(s)$ is the Riemann zeta function.
We also have the integral expression
\begin{align}
\label{eq:Li2int1}
\mathrm{Li}_2(x)=-\int_0^x \frac{\log(1-y)}{y}
dy
=
-\int_0^{-x}\frac{ \log(1+y)}
{y}dy,
\quad
(x\leq 1).
\end{align}
The integrals in \eqref{eq:Li2int1} can be extended on the whole complex plane $\mathbb{C}$.
However, there is a branch point at $x=1$,
 and the function
$\mathrm{Li}_2(x)$ becomes multivalued
on $\mathbb{C}$.
Here we restrict our attention only on the region $x\leq 1$
so that there is no ambiguity of multivaluedness.

The  {\em Rogers dilogarithm\/} $L(x)$ is defined
by
\begin{align}
\label{eq:Rogers2}
L(x)&=
-\frac{1}{2}\int_0^{x} 
\left\{
\frac{\log(1-y)}{y}
+
\frac{\log y}{1-y}
\right\} dy,\quad
(0\leq x\leq 1),
\\
\label{eq:Rogers1}
&= \mathrm{Li}_2(x) + \frac{1}{2}\log x \cdot \log(1-x)
.
\end{align}
Again,
we restrict our attention only on the region $0\leq x \leq 1$,
and there is no ambiguity of multivaluedness.

By \eqref{eq:Li2const1}, we have
\begin{align}
\label{eq:Li2const2}
L(0)=0,
\quad
L(1)=\frac{\pi^2}{6}.
\end{align}
We also have
\begin{align}
\label{eq:Rogers5}
L(x) + L(1-x)= L(1),
\quad
(0\leq x \leq 1).
\end{align}

The following equalities hold:
\begin{align}
\label{eq:Rogers4}
L\left(\frac{x}{1+x}\right)&=
\frac{1}{2}\int_0^{x} 
\left\{
\frac{\log(1+y)}{y}
-
\frac{\log y}{1+y}
\right\} dy,
\quad
(0\leq x),
\\
\label{eq:Rogers3}
&=- \mathrm{Li}_2(-x) - \frac{1}{2}\log x \cdot \log(1+x).
\end{align}
These equalities are less well-known
than \eqref{eq:Rogers2} and \eqref{eq:Rogers1}, but they are crucial for our purpose.
They are easily proven by taking the derivative.
Since the function $x/(1+x)$ is monotonic on $\mathbb{R}_{\geq 0}$
and yields a bijection
from $\mathbb{R}_{\geq 0}$ to $[0,1)$, one may view
\eqref{eq:Rogers4} and \eqref{eq:Rogers3} as an alternative
definition of $L(x)$ on $[0,1]$, where $L(1)$ is defined by the limit
 $x\rightarrow \infty$ in \eqref{eq:Rogers4}.

In view of the form \eqref{eq:Rogers4},
it is  useful to rephrase
the equality \eqref{eq:Rogers5} as 
\begin{align}
\label{eq:Rogers6}
L\left(\frac{x}{1+x}\right) + L\left(\frac{1}{1+x}\right)= L(1),
\quad
(0\leq x ).
\end{align}
Note that there is a duality of the variables in \eqref{eq:Rogers6}:
\begin{align}
\label{eq:dual1}
\frac{1}{1+x}=\frac{x}{1+x}\bigg|_{x=x^{-1}}.
\end{align}

\subsection{Constancy condition}
\label{subsec:const1}
Let $G$ be any multiplicative abelian group.
Let $G\otimes G$ be its tensor product over $\mathbb{Z}$,
that is, the additive abelian group with generators $f\otimes g$ ($f,g\in G$) and relations
\begin{align}
\label{eq:tensor1}
(fg)\otimes h = f\otimes h + g\otimes h,
\quad
f\otimes (gh) = f\otimes g + f\otimes h.
\end{align}
Note that $1\otimes h = h \otimes 1 = 0$.
Let $S^2 G$ be the {\em symmetric subgroup} of
$G\otimes G$,
namely, the subgroup generated by all $f\otimes g + g \otimes f$ ($f,g\in G)$.
We define the {\em wedge product} of $G$ as
$\bigwedge^2 G= G\otimes G/
S^2 G$.

Let
$\mathcal{C}=\mathcal{C}([0,1],\mathbb{R}_+)$ be the set of all 
 positive-real-valued and differentiable functions on the interval $[0,1]$ in $\mathbb{R}$.
Regarding it as 
a multiplicative abelian group, we have
 the wedge product 
$\bigwedge^2\mathcal{C}$.

The following theorem is the starting point of deducing identities 
satisfied by the Rogers dilogarithm.
\begin{thm}[{\cite[Proposition 1]{Frenkel95}}]
\label{thm:const1}
Let $f_1$, \dots, $f_m \in \mathcal{C}$ be  differentiable functions on the interval $[0,1]$
which especially takes values in the interval $(0,1)$.
Suppose that they satisfy the following  relation in $\bigwedge^2\mathcal{C}$:
 (Constancy condition)
\begin{align}
\label{eq:const1}
\sum_{t=1}^m f_t \wedge (1-f_t) = 0.
\end{align}
Then, the  sum of the Rogers dilogarithm
\begin{align}
\label{eq:sum1}
\sum_{t=1}^m L(f_t(u))
\end{align}
is constant as a function of $u\in [0,1]$.
\end{thm}

In view of the form \eqref{eq:Rogers4}, it is useful to rephrase
Theorem \ref{thm:const1}  as
follows:

\begin{thm}[cf. {\cite[Equation (2.3)]{Chapoton05}}]
\label{thm:const2}
Let $y_1$, \dots, $y_m \in \mathcal{C}$.
Suppose that they satisfy the following the relation in $\bigwedge^2\mathcal{C}$:
 (Constancy condition)
\begin{align}
\label{eq:const2}
\sum_{t=1}^m y_t \wedge (1+y_t) = 0.
\end{align}
Then, the  sum of the Rogers dilogarithm
\begin{align}
\label{eq:sum2}
\sum_{t=1}^m L
\left(
\frac{y_t(u)}
{1+y_t(u)}
\right)
\end{align}
is constant as a function of $u\in [0,1]$.
\end{thm}
\rm 

\subsection{Seed mutations}
To make use of wonderful Theorem \ref{thm:const2}, 
it is essential to find a family of functions
$y_1,\dots,y_m\in \mathcal{C}$
 which satisfy the constancy condition
\eqref{eq:const2}.
This is  where cluster algebras take part.
See  \cite{Fomin07} for a general reference on cluster algebras.


First we recall the notion of a semifield, following \cite{Fomin07}.

\begin{defn}
A multiplicative abelian group $\mathbb{P}$
is called a {\em semifield} if it is endowed with a binary operation $\oplus$
which is commutative, associative, and distributative,
i.e., $a(b\oplus c)=ab \oplus ac$. The operation $\oplus$ is called the {\em addition}.
\end{defn}

Here we mainly use the following examples.

\begin{ex}
\label{ex:semi1}
 (1) The set $\mathbb{R}_+$ of all positive real numbers is a semifield,
where the product and the addition are given by the usual ones.
\par
(2)
Let $y=(y_i)_{i=1}^n$ be an $n$-tuple of formal commutative variables.
We say that a rational function $f(y)$ in $y$
with coefficients in $\mathbb{Q}$ {\em has a subtraction-free expression}
if it is represented as $f(y)=P(y)/Q(y)$ such that $P(y)$ and $Q(y)$ are nonzero polynomials
in $y$ with nonnegative integer coefficients.
Let  $\mathbb{Q}_+(y)$ be  the set of all rational functions in $y$ having
subtraction-free expressions.
Then, $\mathbb{Q}_+(y)$
 is a semifield,
where the product and the addition are given by the usual ones for rational functions.
We call it  the {\em universal semifield} of $y$.
\par
(3)
Let $y=(y_i)_{i=1}^n$ be an $n$-tuple of formal commutative variables.
Let 
\begin{align}
\mathrm{Trop}(y)
=\Bigr\{ \prod_{i=1}^n y_i^{a_i}
\mid
a_i \in \mathbb{Z}
\Bigl\}.
\end{align}
Then, $\mathrm{Trop}(y)$
 is a semifield,
where the product  is given by the usual one
 for Laurent monomials,
 while the addition is given by the following {\em tropical sum}:
 \begin{align}
 \prod_{i=1}^n y_i^{a_i}
 \oplus
 \prod_{i=1}^n y_i^{b_i}
 =
 \prod_{i=1}^n y_i^{\min(a_i,b_i)}.
 \end{align}
We call it  the {\em tropical semifield} of $y$.
\end{ex}

Let $\mathbb{P}$ be any semifield.
By regarding it as an abelian multiplicative group,
let $\mathbb{Z}[\mathbb{P}]$ be its group ring.
It is known that $\mathbb{Z}[\mathbb{P}]$ is a domain \cite{Fomin02},
i.e., there is no divisor.
Thus, the field of the fractions of $\mathbb{Z}[\mathbb{P}]$ is well-defined,
and it is denoted by $\mathbb{Q}(\mathbb{P})$ here.

We say that an (integer) square matrix $B$ is {\em skew-symmetrizable}
if there is a diagonal matrix $R=\mathrm{diag}(r_1,\dots,r_n)$ of the same size with positive (integer) diagonal entries $r_1,\dots,r_n$
such that $RB$ is skew-symmetric, i.e., $(RB)^T=-RB$.
Also we  call such $R$ a {\em skew-symmetrizer} of $B$.
Note that if $RB$ is skew-symmetric,
then $BR^{-1}$ is also skew-symmetric.

\par

Now let us give two most important notions in cluster algebras,
namely, a {\em seed} and its {\em mutation}.
(We do not give the definition of a {\em cluster algebra} itself,
because it is not essential in this paper.)

\begin{defn}
\label{defn:mut1}
Let us fix a positive integer $n\in \mathbb{Z}_+$ and a semifield $\mathbb{P}$,
 which are called the {\em rank} and the {\em coefficient semifield} (of a cluster algebra
under consideration).
Let $w=(w_i)_{i=1}^n$ be an $n$-tuple of formal commutative variables,
and let $\mathcal{F}=(\mathbb{Q}(\mathbb{P}))(w)$ be the rational function field of 
$w$ with coefficients in $\mathbb{Q}(\mathbb{P})$.
We call $\mathcal{F}$ the {\em ambient field}.
\par
(1).
Let $(B,x,y)$ be a triplet such that
\begin{itemize}
\item
$B=(b_{ij})_{i,j=1}^n$ is a skew-symmetrizable integer matrix  of size $n$,
\item
$x=(x_i)_{i=1}^n$ is an $n$-tuple 
of algebraically independent elements of  $\mathcal{F}$,
\item
$y=(y_i)_{i=1}^n$ is an $n$-tuple  of elements of $\mathbb{P}$.
\end{itemize}
We call such $(B,x,y)$ a {\em seed},
and call $B$, $x$, and $y$  the {\em exchange matrix},
 the {\em $x$-variables}, and the {\em $y$-variables} of a seed $(B,x,y)$,
respectively.
Also, we set
\begin{align}
\label{eq:yhat1}
\hat{y}_i
&
=
y_i \prod_{j=1}^n x_{j}^{b_{ji}}
\in \mathcal{F},
\quad
i=1,\dots,n,
\end{align}
and call them  the {\em $\hat{y}$-variables} of a seed $(B,x,y)$.
\par
(2).
For any seed $(B,x,y)$ and any $k=1,\dots,n$, we define a new seed
$(B',x',y')=\mu_k(B,x,y)$, called the {\em mutation of $(B,x,y)$ at $k$}, as follows:
\begin{align}
\label{eq:Bmut1}
b'_{ij}&=
\begin{cases}
-b_{ij} & \text{$i=k$ or $j=k$}\\ 
b_{ij}+[-b_{ik}]_+b_{kj} + b_{ik}[b_{kj}]_+
& i,j\neq k,
\end{cases}
\\
\label{eq:xmut1}
x'_i&=
\begin{cases}
\displaystyle
x_k^{-1}
\left(\prod_{j=1}^n
x_j^{[-b_{jk}]_+}
\right)
 \frac{1+ \hat{y}_k}{1\oplus y_k}
& i=k\\ 
x_i 
& i\neq k,
\end{cases}
\\
\label{eq:ymut1}
y'_i&=
\begin{cases}
y_k^{-1}& i=k\\ 
y_i y_k^{[b_{ki}]_+} (1\oplus y_k)^{-b_{ki}}
& i\neq k.
\end{cases}
\end{align}
Here and elsewhere, for any integer $a$,
we define
\begin{align}
[a]_+=\max(a,0).
\end{align}
\end{defn}

The following facts can be easily checked:

(1). The mutation $\mu_k$ is an involution,
i.e., $\mu_k(\mu_k(B,x,y))=(B,x,y)$.

(2). If $R$ is a skew-symmetrizer of $B$,
then $R$ is also a skew-symmetrizer of $B'$.

(3). The $\hat{y}$-variables  transform  in $\mathcal{F}$
as the $y$-variables;
namely,
\begin{align}
\label{eq:yhatmut1}
\hat{y}'_i&=
\begin{cases}
\hat{y}_k^{-1}& i=k\\ 
\hat{y}_i \hat{y}_k^{[b_{ki}]_+} (1+ \hat{y}_k)^{-b_{ki}}
& i\neq k.
\end{cases}
\end{align}


\subsection{Dilogarithm identities}
In this section we specialize the coefficient semifield $\mathbb{P}$
in Definition \ref{defn:mut1}
as $\mathbb{P}=\mathbb{Q}_+(y)$ in Example \ref{ex:semi1} (2) with generators
$y=(y_i)_{i=1}^n$.
Let us choose a seed
$(B,x,y)$, where $B$ and $x$ are arbitrary,
but we especially choose $y$ to be
the generators of $\mathbb{Q}_+(y)$.
Let us call $(B,x,y)$ the {\em initial seed},
and consider a sequence of mutations starting from it:
\begin{align}
\begin{split}
\label{eq:seq1}
(B,x,y)=(B[1],x[1],y[1])&
\buildrel  \mu_{k_1} \over \rightarrow
(B[2],x[2],y[2])
\buildrel  \mu_{k_2} \over \rightarrow\\
&
\cdots
\buildrel  \mu_{k_{m}} \over \rightarrow
(B[m+1],x[m+1],y[m+1]).
\end{split}
\end{align}

Let $R=\mathrm{diag}(r_1,\dots,r_n)$ be a common symmetrizer
of $B[1]$, \dots, $B[m+1]$, 
and let $r$ be the least common multiple of $r_1,\dots,r_n$.
We set $\tilde{r}_i=r /r_i\in \mathbb{Z}_+$.

\begin{defn}
We say that the sequence \eqref{eq:seq1} is {\em $\sigma$-periodic}
for a permutation $\sigma$ of $\{1,\dots,n\}$ if
\begin{align}
b_{\sigma(i)\sigma(j)}[m+1]
=b_{ij},
\quad
x_{\sigma(i)}[m+1]
=x_{i},
\quad
y_{\sigma(i)}[m+1]
=y_{i},
\quad
(i,j=1,\dots,n).
\end{align}
\end{defn}

The following fact connects  cluster algebras and dilogarithm identities.
\begin{thm}[{\cite[Proposition 6.3 and Section 6.5]{Nakanishi10c}}]
\label{thm:yconst1}
Suppose that the sequence \eqref{eq:seq1} is $\sigma$-periodic for some
permutation $\sigma$.
Then, it satisfies the following ``constancy condition" in
the wedge product $\bigwedge^2 \mathbb{Q}_+(y)$,
where we regard $\mathbb{Q}_+(y)$ as a multiplicative abelian group:
\begin{align}
\label{eq:const3}
\sum_{t=1}^{m} \tilde{r}_{k_t} 
\left(y_{k_t}[t] \wedge (1+y_{k_t}[t]) 
\right)= 0.
\end{align}
\end{thm}

Combining Theorems \ref{thm:const2} and \ref{thm:yconst1},
we obtain the first half of the 
 dilogarithm identity associated with
any period of seeds.

\begin{thm}[{\cite[Theorems 6.4 and 6.8]{Nakanishi10c}}]
\label{thm:id1}
Suppose  that the sequence \eqref{eq:seq1} is $\sigma$-periodic for some
permutation $\sigma$.
Let 
\begin{align}
\label{eq:phi1}
\varphi: \mathbb{Q}_+(y) \rightarrow \mathbb{R}_+
\end{align}
be any semifield homomorphism.
Then, the following sum does not depend on the choice of $\varphi$:
\begin{align}
\label{eq:id4}
\sum_{t=1}^{m}
\tilde{r}_{{k}_t} L
\left(
\varphi
\left(
\frac{y_{k_t}[t]}
{1+y_{k_t}[t]}
\right)
\right).
\end{align}
\end{thm}

\begin{rem}
Theorem 6.8 in \cite{Nakanishi10c} was proved under the assumption
of the sign-coherence property of the $c$-vectors,
which  is now proved by \cite{Gross14}.
See Theorem \ref{thm:sign1} below.
\end{rem}

The second half of the dilogarithm identity is about the constant value of \eqref{eq:id4}.
To describe it,
we introduce the semifield homomorphism ({\em tropicalization map})
\begin{align}
\begin{matrix}
\pi: &  \mathbb{Q}_+(y) & \rightarrow &\mathrm{Trop}(y)\\
& y_i & \mapsto & y_i.
\end{matrix}
\end{align}
We then apply the map $\pi$ to each $y$-variable $y_i[t]$
in \eqref{eq:seq1}, and express it as
\begin{align}
\label{eq:tropy1}
\pi(y_i[t])=\prod_{j=1}^n y_j^{c_{ji}[t]}.
\end{align}
Thus, we have a family of square matrices $C[t]=(c_{ij}[t])_{i,j=1}^n$ for
$t=1,\dots,m+1$,
which are called the {\em $C$-matrices} for the sequence \eqref{eq:seq1}.
Alternatively, they can be directly defined through the following system of recursion relations
\cite{Fomin07}:
\par
(initial condition)
\begin{align}
c_{ij}[1]=\delta_{ij},
\end{align}
\par
(recursion relation)
\begin{align}
c_{ij}[t+1]=
\begin{cases}
-c_{ik_t}[t]
& j=k_t\\
c_{ij}[t]
+
[-c_{ik_t}[t]]_+ b_{k_tj}[t]
+ c_{ik_t}[t]
[b_{k_tj}[t]]_+
&
j\neq k_t.
\end{cases}
\end{align}
The $i$th column vector $c_i[t]=(c_{ji}[t])_{j=1}^n$ of the matrix $C[t]$
is called the {\em $c$-vector} of $y_i[t]$.
By the definition of  \eqref{eq:tropy1},
it is the ``exponent vector" of
the tropical $y$-variable $\pi(y_i[t])$.

The following property is fundamental in the theory of cluster algebras,
and it is
originally conjectured by \cite{Fomin07} and proved in full generality recently:

\begin{thm}[(Sign coherence) {\cite[Corollary 5.5]{Gross14}}]
\label{thm:sign1}
Each $c$-vector $c_i[t]$ is a nonzero vector, and all its components
are either nonnegative or nonpositive.
\end{thm}

Accordingly, we set the {\em tropical sign\/}
$\varepsilon(y_i[t])\in \{ \pm 1 \}$ of $y_i[t]$ as $1$ (resp. $-1$) if 
all components of  $c_i[t]$ is nonnegative (resp. nonpositive).
For simplicity, let us write
\begin{align}
\varepsilon_t=\varepsilon(y_{k_t}[t]).
\end{align}

Now,
continuing from
Theorem \ref{thm:id1},
we can state the second half of the dilogarithm identity.
\begin{thm}[{\cite[Theorems 6.4 and 6.8]{Nakanishi10c}}]
\label{thm:id2}
Under the assumption of Theorem \ref{thm:id1},
we have the following equality for any choice of $\varphi$ in \eqref{eq:phi1}:
\begin{align}
\label{eq:id5}
\sum_{t=1}^{m}
\tilde{r}_{{k}_t} L
\left(
\varphi
\left(
\frac{y_{k_t}[t]}
{1+y_{k_t}[t]}
\right)
\right)
=
\sum_{t=1}^{m}
\frac{1-\varepsilon_t}{2}
\tilde{r}_{{k}_t} L(1).
\end{align}
\end{thm}
Using \eqref{eq:Rogers6},
we also have an alternative form of 
the identity \eqref{eq:id5},
which is constant-term free.

\begin{thm}[{\cite[Theorem 2.9]{Kashaev11}}]
\label{thm:id21}
The identity  \eqref{eq:id5} is equivalent to the following one:
\begin{align}
\label{eq:id6}
\sum_{t=1}^{m}
\varepsilon_t
\tilde{r}_{{k}_t} L
\left(
\varphi
\left(
\frac{(y_{k_t}[t])^{\varepsilon_t}}
{1+(y_{k_t}[t])^{\varepsilon_t}}
\right)
\right)
=
0.
\end{align}
\end{thm}


We are going to generalize
Theorems
\ref{thm:id2} and \ref{thm:id21}
based on generalized cluster algebras.

\section{Rogers dilogarithms of higher degree}

In this section we introduce the Rogers dilogarithms of higher degree.
Then, we prove the constancy condition
which is parallel to Theorem \ref{thm:const2}.

\subsection{Rogers dilogarithms of higher degree}
Let $d$ be any positive integer, and let $z=(z_s)_{s=0}^{d}$
such that $z_0=z_d=1$ and $z_1,\dots,z_{d-1}\in \mathbb{R}_{\geq 0}$
are arbitrary.
Let $P_{d,z}(x)$ be
the polynomial 
in a single variable $x$
defined by
\begin{align}
\label{eq:poly1}
P_{d,z}(x)=\sum_{s=0}^d z_s x^s.
\end{align}
Below we assume the following {\em generic condition} for $z$:
\begin{align}
\label{eq:generic1}
\text{The polynomial $P_{d,z}(x)$
has no root  on $\mathbb{R}$ except for $x=-1$.}
\end{align}
In \cite{Nakanishi14a}
 the {\em (Euler) dilogarithm $\mathrm{Li}_{2;d,z}(x)$
 of degree $d$ with coefficients $z$} is defined as follows
(cf. \eqref{eq:Li2int1}):
\begin{align}
\label{eq:hLi2int1}
\mathrm{Li}_{2;d,z}(x)=
-\int_0^{-x} \frac{\log P_{d,z}(y)}{y}dy,
\quad
(x\leq 1).
\end{align}
By the condition \eqref{eq:generic1},
there is no ambiguity of multivaluedness 
in the region $x\leq 1$.

In view of \eqref{eq:Rogers4} and \eqref{eq:Rogers3}
we define the {\em Rogers dilogarithm $L_{d,z}(x)$ of degree $d$
with coefficients $z$} as follows:
\begin{align}
\label{eq:Rogers8}
L_{d,z}\left(\frac{x^d}{P_{d,z}(x)}\right)&=
\frac{1}{2}\int_0^{x} 
\left\{
\frac{\log P_{d,z}(y)}{y}
-
\frac{\log y}{P_{d,z}(y)}
P'_{d,z}(y)
\right\}dy,
\quad
(0\leq x),
\\
\label{eq:Rogers7}
&=- \mathrm{Li}_{2;d,z}(-x)
 - \frac{1}{2}\log x\cdot \log P_{d,z}(x),
\end{align}
where $P'_{d,z}(x)$ denotes the derivative of $P_{d,z}(x)$.
Since the function $x^d/P_{d,z}(x)$ is monotonic on $\mathbb{R}_{\geq  0}$
and yields a bijection
from $\mathbb{R}_{\geq  0}$ to $[0,1)$, 
\eqref{eq:Rogers8} and \eqref{eq:Rogers7} unambiguously
determine the function $L_{d,z}(x)$ on $[0,1]$, where $L_{d,z}(1)$ is defined 
 by the limit
 $x\rightarrow \infty$ in \eqref{eq:Rogers8}.

For the above $z=(z_s)_{s=0}^d$,
we define its {\em reverse} $z^*$ as
\begin{align}
z^*=(z^*_s)_{s=0}^d, \quad z^*_s =z_{d-s}.
\end{align}
Clearly, it holds that $(z^*)^*=z$.
Also, we have
\begin{align}
\label{eq:dual3}
P_{d,z^*}(x^{-1})&= x^{-d} P_{d,z}(x),\\
\label{eq:dual4}
P'_{d,z^*}(x)\vert_{x= x^{-1}}&= -x^{2-d} P'_{d,z}(x)
+ d x^{1-d}P_{d,z}(x).
\end{align}
By \eqref{eq:dual3}, $a\neq 0$ is a root of $P_{d,z^*}(x)$ if and only if
$a^{-1}$ is a root of $P_{d,z}(x)$.
Thus, $z^*$ also satisfies the generic condition \eqref{eq:generic1}.

The following is the counterpart of the equalities \eqref{eq:Rogers6}
and \eqref{eq:dual1}:
\begin{prop}
We have the equalities
\begin{gather}
\label{eq:Rogers9}
L_{d,z}\left(\frac{x^d}{P_{d,z}(x)}\right) + L_{d,z^*}\left(\frac{1}{P_{d,z}(x)}\right)=
 L_{d,z}(1)=L_{d,z^*}(1),
\quad
(0\leq x ),\\
\label{eq:dual2}
\frac{1}{P_{d,z}(x)}=\frac{x^d}{P_{d,z^*}(x)}\Bigr|_{x=x^{-1}}.
\end{gather}
\end{prop}
\begin{proof}
The equality \eqref{eq:dual2} is immediate from \eqref{eq:dual3}.
Let us show \eqref{eq:Rogers9}.
First, we show that the left-hand side of \eqref{eq:Rogers9} does not
depend on $x$. To do it, we  apply \eqref{eq:dual2}
in the second term of the left-hand side of \eqref{eq:Rogers9},
then take the derivative with respect to $x$.
Then, by \eqref{eq:Rogers8}, we obtain
\begin{align}
\label{eq:der1}
\begin{split}
\frac{1}{2}
&\left(
\frac{\log P_{d,z}(x)}{x}
-
\frac{\log x}{P_{d,z}(x)}
P'_{d,z}(x)
\right)\\
&
- \frac{1}{2}
x^{-2}
\left(
\frac{\log P_{d,z^*}(x^{-1})}{x^{-1}}
-
\frac{\log x^{-1}}{P_{d,z^*}(x^{-1})}
\left(
P'_{d,z^*}(x)\Bigr|_{x=x^{-1}}
\right)
\right).
\end{split}
\end{align}
Then, using \eqref{eq:dual3} and \eqref{eq:dual4},
it is easy to check that \eqref{eq:der1} vanishes.
Thus, the left-hand side of \eqref{eq:Rogers9} is
a constant $C$ with respect to $x$. Then, taking $x\rightarrow \infty$ in it, we have
 $C=L_{d,z}(1)$,
while setting $x=0$, we have  $C=L_{d,z^*}(1)$.
\end{proof}

Let us also introduce the function $\tilde{L}_{d,z}(x)$ by
\begin{align}
\label{eq:Rogers10}
\tilde{L}_{d,z}(x):=
L_{d,z}\left(\frac{x^d}{P_{d,z}(x)}\right),
\quad
(0\leq x).
\end{align}
Then, by \eqref{eq:dual2}, the equality \eqref{eq:Rogers9} is written as
\begin{gather}
\label{eq:Rogers11}
\tilde{L}_{d,z}(x) + \tilde{L}_{d,z^*}(x^{-1})=
 \tilde{L}_{d,z}(\infty)=\tilde{L}_{d,z^*}(\infty),
\quad
(0\leq x ),
\end{gather}
where $ \tilde{L}_{d,z}(\infty):=\lim_{x\to \infty}  \tilde{L}_{d,z}(x)$.
Though
we are attached to the function $L_{d,z}(x)$,
since it is more directly related to
the classic Rogers dilogarithm $L(x)$,
all results in this paper are more simply described with the function $\tilde{L}_{d,z}(x)$
as \eqref{eq:Rogers11}.
So, from now on, we mainly use the function $\tilde{L}_{d,z}(x)$
instead of $L_{d,z}(x)$.

\subsection{Constancy condition}
\label{subsec:const2}

Recall that
$\mathcal{C}=\mathcal{C}([0,1],\mathbb{R}_+)$ is the set of all 
 positive-real-valued and differentiable functions on the interval $[0,1]$
 as defined in Section \ref{subsec:const1}.

\begin{thm}[cf. Theorem \ref{thm:const2}]
\label{thm:gconst2}
Let $y_1$, \dots, $y_m \in \mathcal{C}$,
and, for $t=1,\dots,m$,
 let $P_{d_t,z_t}(x)$ be
a  degree $d_t$ polynomial in $x$ whose coefficients $z_t$ 
satisfy the generic condition \eqref{eq:generic1}.
Suppose that they satisfy the following  relation in $\bigwedge^2\mathcal{C}$:
 (Constancy condition)
\begin{align}
\label{eq:const4}
\sum_{t=1}^m y_t \wedge P_{d_t,z_t}(y_t) = 0.
\end{align}
Then, the  sum of the Rogers dilogarithms of higher degree
\begin{align}
\label{eq:sum3}
\sum_{t=1}^m \tilde{L}_{d_t,z_t}
\left(y_t(u)
\right)
\end{align}
is constant as a function of $u\in [0,1]$.
\end{thm}
\begin{proof}
The proof essentially repeats the one for Theorem \ref{thm:const1}
due to \cite{Frenkel95},
whose idea originates in
 \cite{Bloch78}.
By \eqref{eq:Rogers8}, for each $t=1, \dots, m$, we have
\begin{align}
\label{eq:dL1}
\begin{split}
\frac{d}{du}
 \tilde{L}_{d_t,z_t}
\left( y_t(u)\right)
&= 
\frac{1}{2}
\big(
\log P_{d_t,z_t}(y_t(u))\cdot \frac{d}{du}
 \log y_t(u)\\
 &\qquad
-
\log y_t(u)\cdot \frac{d}{du}
\log P_{d_t,z_t}(y_t(u))
\big).
\end{split}
\end{align}
On the other hand, by the assumption of \eqref{eq:const4},
\begin{align}
\label{eq:sym1}
\sum_{t=1}^m
y_t \otimes P_{d_t,z_t}(y_t)
=
\sum_{i=1}^k
( g_i \otimes h_i + h_i \otimes g_i)
\end{align}
for some $k\geq 1$ and $g_i, h_i\in \mathcal{C}$.
For any $u,v\in [0,1]$, we have an additive group homomorphism
$\Psi_{u,v}:
\mathcal{C}\otimes\mathcal{C}
\rightarrow \mathbb{R}$,
$f\otimes g \mapsto \log f(u) \cdot
\log g(v)$.
Applying it on 
\eqref{eq:sym1}, we have
\begin{align}
\label{eq:sym2}
\sum_{t=1}^m
\log y_t(u)\cdot \log P_{d_t,z_t}(y_t(v))
=
\sum_{i=1}^k
\big( \log g_i(u) \cdot \log h_i(v) + \log h_i(u) \cdot \log g_i(v)
\big).
\end{align}
Then, taking the derivative for $u$ and setting $v=u$, we have
\begin{align}
\label{eq:sym3}
\begin{split}
&\sum_{t=1}^m
\frac{d}{du}
 \log y_t(u)\cdot \log P_{d_t,z_t}(y_t(u))\\
&=
\sum_{i=1}^k
\left(\frac{d}{du}
  \log g_i(u) \cdot \log h_i(u) + \frac{d}{du}
 \log h_i(u) \cdot \log g_i(u)
\right).
\end{split}
\end{align}
Similarly,
taking the derivative for $v$ and setting $v=u$, we have
\begin{align}
\label{eq:sym4}
\begin{split}
&
\sum_{t=1}^m
 \log y_t(u)\cdot\frac{d}{du}
 \log P_{d_t,z_t}(y_t(u))
 \\
&=
\sum_{i=1}^k
\left(  \log g_i(u) \cdot\frac{d}{du}
\log h_i(u) + \log h_i(u) \cdot\frac{d}{du}
 \log g_i(u)
\right).
\end{split}
\end{align}
Note that the right-hand sides of
\eqref{eq:sym3} and \eqref{eq:sym4} are identical.
Thus, by \eqref{eq:dL1},
\eqref{eq:sym3}, and \eqref{eq:sym4},
we obtain the equality
\begin{align}
\frac{d}{du}
\sum_{t=1}^m
 \tilde{L}_{d_t,z_t}
\left( y_t(u)\right)
=0.
\end{align}
\end{proof}

\section{Identities associated with periods of generalized cluster algebras}

In this section we show Statement (II') in Section 1;
that is,
there is a dilogarithm identity of higher degree associated with any period of seeds
of a generalized cluster algebra.

\subsection{Seed mutations for generalized cluster algebras}
\label{subsec:seed2}

Generalized cluster algebras were introduced by \cite{Chekhov11}.
Here we present the definitions
of a {\em seed} and its {\em mutation} for a generalized cluster algebra,
following \cite{Nakanishi14a,Nakanishi15}.
(Again, we do not give the definition of a {\em generalized cluster algebra} itself,
because it is not essential in this paper.)

\begin{defn}
\label{defn:mut2}
Let us fix a positive integer $n$ and a semifield $\mathbb{P}$,
 which are called the {\em rank} and the {\em coefficient semifield} (of a generalized
 cluster algebra
under consideration).
In addition, we also fix an $n$-tuple $d=(d_i)_{i=1}^n$
of positive integers, which is called the {\em mutation degree}.
For the simplest choice $d=(1,\dots,1)$,
 it reduces to the ordinary cluster algebra case.
Let $w=(w_i)_{i=1}^n$ be an $n$-tuple of formal commutative variables
and let $\mathcal{F}=(\mathbb{Q}(\mathbb{P}))(w)$ be the rational function field of 
$w$ with coefficients in $\mathbb{Q}(\mathbb{P})$,
which is called the {\em ambient field}.
\par
(1).
Let $(B,x,y,z)$ be a quartet such that
\begin{itemize}
\item
$B=(b_{ij})_{i,j=1}^n$ is a skew-symmetrizable integer matrix  of size $n$,
\item
$x=(x_i)_{i=1}^n$ is an $n$-tuple 
of algebraically independent elements of  $\mathcal{F}$,
\item
$y=(y_i)_{i=1}^n$ is an $n$-tuple  of elements of $\mathbb{P}$.
\item
$z=(z_{i,s}\mid i=1,\dots,n; s=0,\dots,d_i)$ is a collection  of elements of $\mathbb{P}$
such that $z_{i,0}=z_{i,d_i}=1$ for any $i=1,\dots,n$.
\end{itemize}
Here we call such $(B,x,y,z)$ a {\em seed},
and call $B$, $x$,  $y$, and $z$  the {\em exchange matrix},
 the {\em $x$-variables},  the {\em $y$-variables},
 and  {\em $z$-variables} of a seed $(B,x,y,z)$,
respectively.
Also, we set
\begin{align}
\label{eq:gyhat1}
\hat{y}_i
&
=
y_i \prod_{j=1}^n x_{j}^{b_{ji}}
\in \mathcal{F},
\quad
i=1,\dots,n
\end{align}
as before,
and call them  the {\em $\hat{y}$-variables} of a seed $(B,x,y,z)$.
\par
(2).
For any seed $(B,x,y,z)$ and any $k=1,\dots,n$, we define a new seed
$(B',x',y',z')=\mu_k(B,x,y,z)$, called the {\em mutation of $(B,x,y,z)$ at $k$}, as follows:
\begin{align}
\label{eq:gBmut1}
b'_{ij}&=
\begin{cases}
-b_{ij} & \text{$i=k$ or $j=k$}\\ 
b_{ij}+d_k([-b_{ik}]_+b_{kj} + b_{ik}[b_{kj}]_+)
& i,j\neq k,
\end{cases}
\\
\label{eq:gxmut1}
x'_i&=
\begin{cases}
\displaystyle
x_k^{-1}
\left(\prod_{j=1}^n
x_j^{[-b_{jk}]_+}
\right)^{d_k}
 \frac{P_{d_k,z_k} (\hat{y}_k)}{P_{d_k,z_k} ({y}_k)\vert_{\mathbb{P}}}
& i=k\\ 
x_i 
& i\neq k,
\end{cases}
\\
\label{eq:gymut1}
y'_i&=
\begin{cases}
y_k^{-1}& i=k\\ 
y_i
\left( y_k^{[b_{ki}]_+} 
\right)^{d_k}
(P_{d_k,z_k} ({y}_k)\vert_{\mathbb{P}})^{-b_{ki}}
& i\neq k.
\end{cases}
\\
\label{eq:gzmut1}
z'_{i,s}&=
z_{i,d_i-s}.
\end{align}
Here and elsewhere, for the $z$-variables $z=(z_{i,s})_{i=1,\dots,n; s=0,\dots,d_i}$,
we set
\begin{align}
z_k:=(z_{k,s})_{s=0}^{d_k},
\end{align}
and
\begin{align}
\label{eq:poly2}
P_{d_k,z_k}(\hat{y}_k)=\sum_{s=0}^{d_k} z_{k,s} \hat{y}_k^s,\quad
\quad
P_{d_k,z_k}(y_k)\vert_{\mathbb{P}}=\bigoplus_{s=0}^{d_k} z_{k,s} y_k^s.
\end{align}
\end{defn}

The following facts can be easily checked:

(1) The mutation $\mu_k$ is an involution,
i.e., $\mu_k(\mu_k(B,x,y,z))=(B,x,y,z)$.

(2)  If $R$ is a skew-symmetrizer of $B$,
then $R$ is also a skew-symmetrizer of $B'$.

(3)
The $\hat{y}$-variables in \eqref{eq:gyhat1}  transform
 in $\mathcal{F}$
as the $y$-variables,
namely,
\begin{align}
\label{eq:gyhatmut1}
\hat{y}'_i&=
\begin{cases}
\hat{y}_k^{-1}& i=k\\ 
\hat{y}_i
\left( \hat{y}_k^{[b_{ki}]_+} 
\right)^{d_k}
(P_{d_k,z_k} (\hat{y}_k))^{-b_{ki}}
& i\neq k.
\end{cases}
\end{align}


\subsection{Dilogarithm identities}
Let $n$ and $d=(d_i)_{i=1}^n$ be  the ones in Section
\ref{subsec:seed2}.
Let $\mathbb{Q}_+(y,z)$
be the universal semifield
 with generators (formal commutative variables)
$y=(y_i)_{i=1}^n$ and $z=(z_{i,s})_{i=1,\dots,n;s=0,\dots,d_i}$
as defined in Example \ref{ex:semi1} (2).
However,
in accordance with our situation, 
$z_{i,0}$ and $z_{i,d_i}$ $(i=1,\dots,n)$
are specialized to the identity element $1$.
For example, in the case $d=(1,\dots,1)$, we have $\mathbb{Q}_+(y,z)
=\mathbb{Q}_+(y)$.

From now on we specialize the coefficient semifield $\mathbb{P}$
in Definition \ref{defn:mut2}
as $\mathbb{P}=\mathbb{Q}_+(y,z)$.
Let us choose a seed $(B,x,y,z)$,
where $B$ and $x$ are arbitrary,
but we especially choose $y$ and $z$ to be
the generators $y$ and $z$ of $\mathbb{Q}_+(y,z)$.
Let us call $(B,x,y,z)$ the {\em initial seed},
and consider a sequence of mutations starting from it:
\begin{align}
\label{eq:gseq1}
\begin{split}
(B,x,y,z)=&(B[1],x[1],y[1],z[1])
\buildrel  \mu_{k_1} \over \rightarrow
(B[2],x[2],y[2],z[2])
\buildrel  \mu_{k_2} \over \rightarrow
\cdots\\
&
\cdots
\buildrel  \mu_{k_{m}} \over \rightarrow
(B[m+1],x[m+1],y[m+1],z[m+1]).
\end{split}
\end{align}

Let $R=\mathrm{diag}(r_1,\dots,r_n)$ be a common skew-symmetrizer
of $B[1]$, \dots, $B[m+1]$,
and let $r$ be the least common multiple of $r_1,\dots,r_n$.
We set $\tilde{r}_i=r /r_i\in \mathbb{Z}_+$.

\begin{defn}
\label{defn:period1}
We say that the sequence \eqref{eq:gseq1} is {\em $\sigma$-periodic}
for a permutation $\sigma$ of $\{1,\dots,n\}$ if
\begin{align}
\label{eq:gperiod1}
\begin{split}
b_{\sigma(i)\sigma(j)}[m+1]
&=b_{ij},
\quad
x_{\sigma(i)}[m+1]
=x_{i},
\quad
y_{\sigma(i)}[m+1]
=y_{i},\quad
(i,j=1,\dots,n).
\end{split}
\end{align}
\end{defn}

\begin{prop}
If the sequence \eqref{eq:gseq1} is {\em $\sigma$-periodic},
then we have
\begin{align}
\label{eq:rsym1}
r_{\sigma(i)}=r_i,
\quad
(i=1,\dots,n).
\end{align}
\end{prop}
\begin{proof}
Without losing generality, one can assume that $B=B[1]$ is decomposed into
a block diagonal form such that each block
is indecomposable.
By \eqref{eq:Bmut1},
mutations preserve the block diagonal form.
Moreover,
by \eqref{eq:ymut1} and the assumption
$y_{\sigma(i)}[m+1]
=y_{i}$,
$\sigma$ only permutes the indices in the same block.
On the other hand, for each indecomposable bock $(b_{ij})_{i,j=p}^q$,
its skew-symmetrizer is unique up to a multiplicative constant.
In particular, there is a unique minimal skew-symmetrizer
$\mathrm{diag}(r'_p,\dots,r'_q)$  of the block.
Then, by the assumption
$b_{\sigma(i)\sigma(j)}[m+1]
=b_{ij}$,
 $\mathrm{diag}(r'_{\sigma(p)},\dots,r'_{\sigma(q)})$
 is also the minimal skew-symmetrizer of the block.
 Therefore,
we have $r'_{\sigma(i)}=r'_i$ ($i=p,\dots,q$).
Thus, $r_{\sigma(i)}=r_i$ ($i=p,\dots,q$) holds.
\end{proof}

The following fact connects generalized
cluster algebras and dilogarithm identities of higher degree.
\begin{thm}[cf. Theorem \ref{thm:yconst1}]
\label{thm:gyconst1}
Suppose that the sequence \eqref{eq:gseq1} is $\sigma$-periodic for some
permutation $\sigma$.
Then, it satisfies the following ``constancy condition" in
the wedge product $\bigwedge^2 \mathbb{Q}_+(y,z)$,
where we regard $\mathbb{Q}_+(y,z)$ as a multiplicative abelian group:
\begin{align}
\label{eq:gconst3}
\sum_{t=1}^{m} \tilde{r}_{k_t}
\left( y_{k_t}[t] \wedge P_{d_{k_t},z_{k_t}[t]}(y_{k_t}[t])
\right) = 0.
\end{align}
\end{thm}
Our proof of  Theorem \ref{thm:gyconst1} relies on
some detailed results on
the $y$-variables obtained in \cite{Nakanishi14a}, and it will  be given in Section \ref{sec:proof1}.

Admitting Theorem \ref{thm:gyconst1},
we give the first half of the dilogarithm
identity of higher degree associated with
any period of seeds.

\begin{thm}[cf. Theorem \ref{thm:id1}]
\label{thm:gid1}
Suppose that the sequence \eqref{eq:gseq1} is $\sigma$-periodic for some
permutation $\sigma$.
Let 
\begin{align}
\label{eq:gphi1}
\varphi: \mathbb{Q}_+(y,z) \rightarrow \mathbb{R}_+
\end{align}
be any semifield homomorphism
such that 
$\varphi(z_{i})=(\varphi(z_{i,s}))_{s=0}^{d_{i}}$
satisfies the generic condition \eqref{eq:generic1}
for any $i=1,\dots,n$.
Then, the following sum only depend on
the images of the  $z$-variables  $\varphi(z_{i,s})$,
and does not depend on the images 
of the  $y$-variables $\varphi(y_i)$:
\begin{align}
\label{eq:gid4}
\sum_{t=1}^{m}
\tilde{r}_{{k}_t} \tilde{L}_{d_{k_t},\varphi(z_{k_t}[t])}
\left(
\varphi
\left(
y_{k_t}[t]
\right)
\right).
\end{align}
\end{thm}

\begin{proof} The argument is standard (e.g., the proof of \cite[Theorem 6.4]{Nakanishi10c}).
Suppose that there are two semifield homomorphisms
$\varphi_0$ and $\varphi_1$
from $ \mathbb{Q}_+(y,z) $ to $\mathbb{R}_+$ such that
$\varphi_0(z_{i,s})=\varphi_1(z_{i,s})$.
Then, we can interpolate them by a family of 
semifield homomorphisms $\varphi_u$ ($u\in [0,1]$),
\begin{align}
\label{eq:gphi2}
\begin{matrix}
\varphi_u& : \mathbb{Q}_+(y,z)&  \rightarrow & \mathbb{R}_+\\
& y_i & \mapsto & (1-u) \varphi_0(y_i) + u \varphi_1 (y_i)\\
& z_{i,s} & \mapsto & \varphi_0(z_{i,s})=\varphi_1(z_{i,s}).
\end{matrix}
\end{align}
Let us introduce positive-real-valued and differentiable functions
$Y_t \in \mathcal{C}$ ($t=1,\dots,m$)  on the interval $[0,1]$ defined by
\begin{align}
Y_t(u)=\varphi_u(y_{k_t}[t]).
\end{align}
Applying the family of homomorphisms $\varphi_u$ ($u\in [0,1]$) to 
\eqref{eq:gconst3},
we have the constancy condition in $\bigwedge^2 \mathcal{C}$: 
\begin{align}
\label{eq:gconst4}
\sum_{t=1}^{m} \tilde{r}_{k_t}
\left( Y_t \wedge P_{d_k,\varphi_0(z_{k_t}[t])}(Y_t)
\right)
 = 0,
\end{align}
where we used the fact that $\varphi_u(z_{k_t}[t])=\varphi_0(z_{k_t}[t])$
for any $u\in [0,1]$.
Then,
by Theorem \ref{thm:gconst2},
the  dilogarithm sum
\begin{align}
\label{eq:sum4}
\sum_{t=1}^m 
 \tilde{r}_{k_t} 
 \tilde{L}_{d_t,\varphi_0(z_{k_t}[t])}
\left(Y_t(u)
\right)
\end{align}
is constant as a function of $u\in [0,1]$.
In particular, setting $u=0$ and $1$, we have
\begin{align}
\label{eq:sum5}
\sum_{t=1}^m 
 \tilde{r}_{k_t} 
\tilde{L}_{d_t,\varphi_0(z_{k_t}[t])}
\left(\varphi_0(y_{k_t}[t])
\right)
=
\sum_{t=1}^m 
 \tilde{r}_{k_t} 
\tilde{L}_{d_t,\varphi_0(z_{k_t}[t])}
\left(\varphi_1(y_{k_t}[t])
\right),
\end{align}
which is the desired result.
\end{proof}

The second half of the dilogarithm identity of higher degree
is about the constant value of \eqref{eq:gid4}.
To describe it,
we introduce the semifield homomorphism ({\em tropicalization map})
\begin{align}
\label{eq:gtrop1}
\begin{matrix}
\pi: &  \mathbb{Q}_+(y,z) & \rightarrow &\mathrm{Trop}(y,z)\\
& y_i & \mapsto & y_i\\
& z_{i,s} & \mapsto & z_{i,s}.
\end{matrix}
\end{align}
Here,  $\mathrm{Trop}(y,z)$
is the tropical semifield
 with generators 
$y=(y_i)_{i=1}^n$ and $z=(z_{i,s})_{i=1,\dots,n;s=0,\dots,d_i}$
as defined in Example \ref{ex:semi1} (3),
but again
$z_{i,0}$ and $z_{i,d_i}$ $(i=1,\dots,n)$
are specialized to the identity element $1$.

We then apply the map $\pi$ to each $y$-variable $y_i[t]$
in \eqref{eq:gseq1}.
It is known that the image $\pi(y_i[t])$ does not depend on $z$-variables $z$
\cite[Lemma 3.6]{Nakanishi14a};
thus,  it is expressed as
\begin{align}
\label{eq:gtropy1}
\pi(y_i[t])=\prod_{j=1}^n y_j^{c_{ji}[t]}.
\end{align}
Thus, we have a family of square matrices $C[t]=(c_{ij}[t])_{i,j=1}^n$
for $t=1,\dots,m+1$,
which are called the {\em $C$-matrices} for the sequence \eqref{eq:gseq1}.
Alternatively, they can be directly defined through the following system of recursion relations
\cite[Propostition 3.8]{Nakanishi14a}:
\par
(initial condition)
\begin{align}
&\text{}
\quad
c_{ij}[1]=\delta_{ij},
\end{align}
\par
(recursion relation)
\begin{align}
\label{eq:crec1}
c_{ij}[t+1]=
\begin{cases}
-c_{ik_t}[t]
& j=k_t\\
c_{ij}[t]
+
d_{k_t}
([-c_{i{k_t}}[t]]_+ b_{k_t j}[t]
+ c_{ik_t}[t]
[b_{k_tj}[t]]_+)
&
j\neq k_t.
\end{cases}
\end{align}
The $i$th column vector $c_i[t]=(c_{ji}[t])_{j=1}^n$ of the matrix $C[t]$
is called the {\em $c$-vector} of $y[t]_i$.
By the definition of  \eqref{eq:gtropy1},
it is
the ``exponent vector" of
the tropical $y$-variable $\pi(y_i[t])$.

It is known that the sign-coherence property of the $c$-vectors still holds
 for generalized cluster algebras.

\begin{thm}[{\cite[Theorem 3.20]{Nakanishi14a}}]
\label{thm:gsign1}
Each $c$-vector $c_i[t]$ is a nonzero vector, and all its components
are either nonnegative or nonpositive.
\end{thm}

Accordingly, we set the {\em tropical sign\/}
$\varepsilon(y_i[t])\in \{ \pm 1 \}$ of $y_i[t]$ as $1$ (resp. $-1$) if 
all components of  $c_i[t]$ is nonnegative (resp. nonpositive).
For simplicity, let us write
\begin{align}
\label{eq:ts1}
\varepsilon_t=\varepsilon(y_{k_t}[t]).
\end{align}

Now,
continuing from
Theorem \ref{thm:gid1},
we can state the second half of the dilogarithm identity
of higher degree.

\begin{thm}[cf. Theorem \ref{thm:id2}]
\label{thm:gid2}
Under the assumption of Theorem \ref{thm:gid1},
we have the following equality for any choice of $\varphi$ in \eqref{eq:gphi1}:
\begin{align}
\label{eq:gid5}
\sum_{t=1}^{m}
\tilde{r}_{{k}_t} 
 \tilde{L}_{d_{k_t},\varphi(z_{k_t}[t])}
\left(
\varphi
\left(
y_{k_t}[t]
\right)
\right)
=
\sum_{t=1}^{m}
\tilde{r}_{{k}_t} 
\frac{1-\varepsilon_t }{2}
 \tilde{L}_{d_{k_t},\varphi(z_{k_t}[t])}(\infty).
\end{align}
\end{thm}
\begin{proof}
Again, the argument is standard (e.g., the proof of \cite[Theorem 6.4]{Nakanishi10c}).
Let $\varphi$ be any such semifield homomorphism.
Then, we consider a family of semifield homomorphism $\varphi_u$ ($u\in (0,1]$)
as follows:
\begin{align}
\label{eq:gphi3}
\begin{matrix}
\varphi_u& : \mathbb{Q}_+(y,z)&  \rightarrow & \mathbb{R}_+\\
& y_i & \mapsto &  u \varphi (y_i)\\
& z_{i,s} & \mapsto & \varphi(z_{i,s}).
\end{matrix}
\end{align}

First, we claim  the following behavior of $y$-variables in the limit $u\rightarrow 0$:
\begin{align}
\label{eq:lim1}
\lim_{u\to 0}
\varphi_u (y_{k_t}[t])
=
\begin{cases}
0 & \varepsilon_t=1\\
\infty & \varepsilon_t=-1.
\end{cases}
\end{align}
This follows from the forthcoming expression of the $y$-variables
\eqref{eq:sep1},
together with 
 the sign-coherence property in Theorem \ref{thm:gsign1}
 and
the fact that
all polynomials $F_j[t](y,z)$ in \eqref{eq:sep1} have the constant term 1
\cite[Proposition 3.19]{Nakanishi14a}.

On the other hand,
by Theorem \ref{thm:gid1},
one can replace
$\varphi(y_{k_t}[t])$ in the left side of \eqref{eq:gid5}
with $\varphi_u(y_{k_t}[t])$ for any $u\in (0,1]$ without changing the sum therein.
Then, by taking the  limit $u\rightarrow 0$,
we obtain the right-hand side of \eqref{eq:gid5}
thanks to \eqref{eq:lim1}.
\end{proof}

Using \eqref{eq:Rogers9},
we also have an alternative form of 
the identity \eqref{eq:gid5},
which is constant-term free.

\begin{thm}[cf. Theorem \ref{thm:id21}]
\label{thm:gid3}
The identity  \eqref{eq:gid5} is equivalent to the following one:
\begin{align}
\label{eq:gid6}
\sum_{t=1}^{m}
\varepsilon_t
\tilde{r}_{{k}_t} 
 \tilde{L}_{d_{k_t},\varphi((z_{k_t}[t])^{\circ})}
\left(
\varphi
\left(
(y_{k_t}[t])^{\varepsilon_t}
\right)
\right)
=
0,
\end{align}
where
\begin{align}
(z_{k_t}[t])^{\circ}
=
\begin{cases}
z_{k_t}[t] & \varepsilon_t=1\\
(z_{k_t}[t])^* & \varepsilon_t=-1.
\end{cases}
\end{align}
\end{thm}


\begin{rem}
One can force to set some of $\varphi(z_{i,s})$ to be 0,
and Theorems \ref{thm:gid2} and \ref{thm:gid3} still hold  by continuity.

\end{rem}
\section{Proof of Theorem \ref{thm:gyconst1}}
\label{sec:proof1}

In this section we give a proof of 
Theorem \ref{thm:gyconst1},
which is the core of this paper.
The proof here uses the same argument
in the ``second proof" of \cite[Proposition 6.3]{Nakanishi10c} therein,
whose idea originates in \cite[Proposition 6.3]{Fock03}
and \cite[Proposition 2.14]{Fock07}.
Interestingly,
even though the statement of 
Theorem \ref{thm:gyconst1} only involves
$y$- and $z$-variables,
the proof  requires {\em $F$-polynomials} as well,
which are the specializations of the
accompanying $x$-variables.

Let us recall the notion of the {\em $F$-polynomials} for generalized cluster algebras
 \cite{Nakanishi14a} in our context.
 
 \begin{defn}
Let us consider the sequence \eqref{eq:gseq1}.
Apply
  the tropicalization map of  \eqref{eq:gtrop1} to
 all $y$-variables involving in the mutation
 of $x$-variables  \eqref{eq:gxmut1}.
Then, it is known \cite[Proposition 3.3]{Nakanishi14a}
 that the resulting $x$-variable $x_i[t]$ is expressed
as a Laurent polynomial $X_i[t](x,y,z)\in \mathbb{Z}[x^{\pm1},y,z]$.
By specializing $x_1=\dots x_n=1$ in $X_i[t](x,y,z)$,
we obtain a polynomial $F_i[t](y,z)\in \mathbb{Z}[y,z]$,
which is called the {\em $F$-polynomial} of $x_i[t]$.
 \end{defn}

We use the following known properties of the $F$-polynomials,
which generalize the ones for ordinary cluster algebras by \cite{Fomin07}.

\begin{prop}
(1). (\cite[Proposition 3.12]{Nakanishi14a})
The $F$-polynomials satisfy the following system of recursion relations:
\par
(initial condition)
\begin{align}
\label{eq:Fp1}
F_{i}[1](y,z)=1,
\end{align}
\par
(recursion relation)
\begin{align}
\label{eq:Frec1}
F_{i}[t+1](y,z)=
\begin{cases}
\displaystyle
F_{k_t}[t](y,z)^{-1}
\left(
\prod_{j=1}^n
y_j^{[-c_{jk_t}[t]]_+}
F_j[t](y,z)^{[-b_{jk_t}[t]]_+}
\right)^{d_{k_t}}\\
\displaystyle
\times
P_{d_{k_t},z_{k_t}[t]}
\left(
\prod_{j=1}^n
y_j^{c_{jk_t}[t]}
F_j[t](y,z)^{b_{jk_t}[t]}
\right)
& i=k_t\\
F_i[t](y,z)
&
i\neq k_t.
\end{cases}
\end{align}
\par
(2). (\cite[Theorem 3.22]{Nakanishi14a})
 The following equality holds: 
 \par
 (Separation formula)
\begin{align}
\label{eq:sep1}
y_i[t]
=
\prod_{j=1}^n y_j^{c_{ji}[t]}
F_j[t](y,z)^{b_{ji}[t]}.
\end{align}
\end{prop}
We note that due to the above recursion,
$F_i[t](y,z)$ are also viewed as  elements in $\mathbb{Q}_+(y,z)$.

We also need the following periodicity property of the $F$-polynomials.

\begin{prop}
Suppose that the sequence \eqref{eq:gseq1} is $\sigma$-periodic for some
permutation $\sigma$.
Then, the $F$-polynomials obey the same periodicity, i.e.,
\begin{align}
\label{eq:Fperiod1}
F_{\sigma(i)}[m+1](y,z)
=F_{i}[1](y,z) =1.
\end{align}
Therefore, $F_{i}[m+1](y,z)=1$ for any $i=1,\dots,n$.
\end{prop}
\begin{proof}
This is because the $F$-polynomials are defined by
the tropicalization and specialization from the $x$-variables
in  \eqref{eq:gseq1}.
\end{proof}

To each seed $\Sigma[t]=(B[t],x[t],y[t],z[t])$ ($t=1,\dots,m+1)$
in the sequence \eqref{eq:gseq1},
we attach the following element in $\bigwedge^2 \mathbb{Q}_+(y,z)$:
\begin{align}
\label{eq:Vt}
V[t]
:=
\sum_{i=1}^n
\tilde{r}_i 
\left( F_i[t] \wedge y_i[t]\right)
+
\frac{1}{2}
\sum_{i,j=1}^n
 b_{ij}[t] \tilde{r}_j 
 \left(
 F_i[t] \wedge F_j[t] 
 \right).
\end{align}
Note that $V[1]=0$ due to the initial condition
\eqref{eq:Fp1}.

The next result is  crucial 
in our proof of Theorem \ref{thm:gyconst1}.
\begin{lem}
\label{lem:Vdif1}
 The following equality holds:
\begin{align}
V[t+1] - V[t]
=
\tilde{r}_{k_t}
\left( y_{k_t}[t] \wedge
P_{d_{k_t},z_{k_t}[t]}(y_{k_t}[t])
\right).
\end{align}
\end{lem}
\begin{proof}
We prove it by the direct and straightforward calculation.
To make the calculation a little more transparent,
we separate the quantity $V[t]$ in \eqref{eq:Vt}
into two parts,
\begin{align}
\label{eq:Vt1}
V_1[t]=
\sum_{i=1}^n
\tilde{r}_i 
\left(F_i[t] \wedge y_i[t]
\right)
,
\quad
V_2[t]
=
\frac{1}{2}
\sum_{i,j=1}^n
 b_{ij}[t] \tilde{r}_j 
 \left(
 F_i[t] \wedge F_j[t]
 \right)
 ,
\end{align}
and calculate the difference $V_i[t+1]-V_i[t]$  separately.
After a careful calculation,
we obtain the following results.
\begin{align}
\begin{split}
\label{eq:Vdif1}
V_1[t+1]-V_1[t]
&=\tilde{r}_{k_t}
\left(
\left(
\prod_{i=1}^n F_i[t]^{b_{i k_t}[t]}
\right)
\wedge
P_{d_{k_t}, z_{k_t}[t]}(y_{k_t}[t])
\right)
\\
&\quad +
d_{k_t}\tilde{r}_{k_t}
\left(
y_{k_t}[t]
\wedge
\left(
\prod_{i=1}^n y_i^{[-c_{i k_t}[t]]_+}
\right)
\right)
\\
& \quad 
 +
\tilde{r}_{k_t}
\left(
y_{k_t}[t]
\wedge
P_{d_{k_t}, z_{k_t}[t]}(y_{k_t}[t])
\right)
,
\end{split}\\
\begin{split}
\label{eq:Vdif2}
V_2[t+1]-V_2[t]
&=
 -
\tilde{r}_{k_t}
\left(
\left(
\prod_{i=1}^n F_i[t]^{b_{i k_t}[t]}
\right)
\wedge
P_{d_{k_t}, z_{k_t}[t]}(y_{k_t}[t])
\right)
\\
&
\quad -d_{k_t}\tilde{r}_{k_t}
\left(
\left(
\prod_{i=1}^n (F_i[t])^{b_{i k_t}[t]}
\right)
\wedge
\left(
\prod_{i=1}^n y_i^{[-c_{i k_t}[t]]_+}
\right)
\right).
\end{split}
\end{align}
To obtain them, we used
\eqref{eq:gBmut1}, \eqref{eq:gymut1},
\eqref{eq:Frec1},
and also the skew-symmetric property
\begin{align}
b_{ij}[t]\tilde{r}_j
=-b_{ji}[t]\tilde{r}_i.
\end{align}
Summing up \eqref{eq:Vdif1} and \eqref{eq:Vdif2}
and using \eqref{eq:sep1},
we have
\begin{align}
\begin{split}
\label{eq:Vdif3}
V[t+1]-V[t]
&=
d_{k_t}\tilde{r}_{k_t}
\left(
\left(
\prod_{i=1}^n y_i^{c_{i k_t}[t]}
\right)
\wedge
\left(
\prod_{i=1}^n y_i^{[-c_{i k_t}[t]]_+}
\right)
\right)
\\
& \quad
+
\tilde{r}_{k_t}
\left(
y_{k_t}[t]
\wedge
P_{d_{k_t}, z_{k_t}[t]}(y_{k_t}[t])
\right).
\end{split}
\end{align}
Then,
the proof of the lemma completes by showing the first term
in the right-hand side of \eqref{eq:Vdif3} vanishes.
Indeed, using the equality $a=[a]_+ - [-a]_+$,
we have
\begin{align}
\label{eq:CC1}
\left(
\prod_{i=1}^n y_i^{c_{i k_t}[t]}
\right)
\wedge
\left(
\prod_{i=1}^n y_i^{[-c_{i k_t}[t]]_+}
\right)
&=
\left(
\prod_{i=1}^n y_i^{[c_{i k_t}[t]]_+}
\right)
\wedge
\left(
\prod_{i=1}^n y_i^{[-c_{i k_t}[t]]_+}
\right).
\end{align}
Then, either the first or the second component in the right-hand side
of \eqref{eq:CC1}
is $1$ due to the sign-coherence property of the $c$-vectors
in Theorem \ref{thm:gsign1}.
Thus, the right-hand side of \eqref{eq:CC1} vanishes as desired.
\end{proof}

Let us complete the proof of Theorem \ref{thm:gyconst1}.
\begin{proof}[Proof of Theorem \ref{thm:gyconst1}]
Due to the assumption of the periodicity 
\eqref{eq:gperiod1} and the resulting periodicities
of the skew-symmetrizer \eqref{eq:rsym1}
 and the $F$-polynomials
\eqref{eq:Fperiod1},
we have the periodicity of $V[t]$, i.e.,
\begin{align}
\label{eq:Vdif4}
V[m+1]=V[1]=0.
\end{align}
On the other hand, by Lemma \ref{lem:Vdif1},
we have
\begin{align}
\label{eq:Vdif5}
V[m+1]=V[m+1] - V[1]=
\sum_{t=1}^m
\tilde{r}_{k_t} 
\left(y_{k_t}[t] \wedge
P_{d_{k_t},z_{k_t}[t]}(y_{k_t}[t])
\right).
\end{align}
Combining \eqref{eq:Vdif4} and \eqref{eq:Vdif5},
we obtain the constancy condition \eqref{eq:gconst3}.
\end{proof}
\section{Examples}

Here we provide three examples of periodicities
in generalized cluster algebras which are not regarded as
the periodicities in ordinary cluster algebras.
In all of them the permutation $\sigma$ 
in Definition \ref{defn:period1}
is the trivial one.
So far, we do not know any example of 
a periodicity with a nontrivial permutation $\sigma$
which is not regarded as
a periodicity in an ordinary cluster algebra.
(For ordinary cluster algebras there are plenty  examples 
of periodicities with 
nontrivial permutations. See e.g. \cite[Theorem 7.1]{Inoue10a}.)

\subsection{Involution periodicity}
Let $(B,x,y,z)$ be any seed with any rank $n$ and any mutation degree $d$.
Due to the involution property of the mutation $\mu_k$,
the mutation sequence
\begin{align}
\label{eq:gseq2}
\begin{split}
(B,x,y,z)=(B[1],x[1],y[1],z[1])&
\buildrel  \mu_{k} \over \rightarrow
(B[2],x[2],y[2],z[2])
\buildrel  \mu_{k} \over \rightarrow
(B[3],x[3],y[3],z[3]).
\end{split}
\end{align}
is {\em periodic}, i.e., $\sigma$-periodic with $\sigma=\mathrm{id}$.
The  data for the associated dilogarithm identity is given as follows:
\begin{gather}
y_k[1]=y_k, \quad y_k[2] = y_k^{-1},\\
\varepsilon_1=1,\quad \varepsilon_2=-1,\\
z_k[1]=z_k, \quad z_k[2] = z_k^{*}.
\end{gather}
Thus, the dilogarithm identity in the form \eqref{eq:gid5} is
\begin{align}
\label{eq:gid7}
\tilde{r}_{{k}} 
 \tilde{L}_{d_{k},z_{k}}
\left(
y_{k}
\right)
+
\tilde{r}_{{k}} 
 \tilde{L}_{d_{k},z_{k}^*}
\left(
y_{k}^{-1}
\right)
=
\tilde{r}_{{k}} 
 \tilde{L}_{d_{k},z_{k}^*}(\infty),
\end{align}
where for notational simplicity
we omit the evaluation homomorphism $\varphi:\mathbb{Q}_+(y,z)
\rightarrow \mathbb{R}_+$  by
identifying $y_k$ and $z_k$  with their images in $\mathbb{R}_+$ by $\varphi$.
The identity \eqref{eq:gid7} coincides with
\eqref{eq:Rogers8}.
Meanwhile, the dilogarithm identity in the form \eqref{eq:gid6} 
becomes trivial.
\begin{align}
\label{eq:gid8}
\tilde{r}_{{k}} 
 \tilde{L}_{d_{k},z_{k}}
\left(
y_{k}
\right)
-
\tilde{r}_{{k}} 
 \tilde{L}_{d_{k},z_{k}}
\left(
y_{k}
\right)
=
0.
\end{align}

\subsection{Six term relation for type $B_2/C_2$}
Let $n=2$ and  $d=(2,1)$.
We consider an initial seed $(B,x,y,z)$ with
\begin{align}
B=
\left(
\begin{matrix}
0 & -1\\
1 & 0
\end{matrix}
\right),
\quad z_1=(1,\alpha,1),
\quad z_2=(1,1).
\end{align}
We choose $r_1=r_2=1$, so that $\tilde{r}_1=\tilde{r}_2=1$.
Then,
the following  sequence of alternative mutations is known to be periodic
\cite[Section 2.3]{Nakanishi14a}:
\begin{align}
\label{eq:gseq3}
\begin{split}
(B,x,y,z)=(B[1],x[1],y[1],z[1])&
\buildrel  \mu_{1} \over \rightarrow
(B[2],x[2],y[2],z[2])
\buildrel  \mu_{2} \over \rightarrow
\\
&
\cdots
\buildrel  \mu_{2} \over \rightarrow
(B[7],x[7],y[7],z[7]).
\end{split}
\end{align}
The  data for the associated dilogarithm identity is given as follows:
\begin{align}
\begin{split}
&
y_1[1]=y_1, \quad y_2[2] = y_2(1+\alpha y_1 + y_1^2),
\quad
y_1[3]=y_1^{-1}(1+y_2+\alpha y_1y_2 + y_1^2 y_2),\\
&y_2[4] = y_1^{-2}y_2^{-1}
(1+2y_2 + y_2^2 + \alpha y_1y_2 + \alpha y_1 y_2^2 + y_1^2 y_2^2),\\
&y_1[5]=y_1^{-1}y_2^{-1}(1+y_2), \quad y_2[6] = y_2^{-1},
\end{split}
\\
&
\varepsilon_1=\varepsilon_2=1,
\quad \varepsilon_3=\cdots =\varepsilon_6=-1,\\
&z_1[1]=z_1[3]=z_1[5]=z_1, \quad z_2[2] =z_2[4]=z_2[6]= z_2.
\end{align}
Then, the dilogarithm identity in the form \eqref{eq:gid6} 
is explicitly written as follows:
\begin{align}
\label{eq:gid9}
\begin{split}
&
 \tilde{L}_{2,z_{1}}
\left(y_1
\right)
+
 \tilde{L}
\left( y_2(1+\alpha y_1 + y_1^2)
\right)
-
 \tilde{L}_{2,z_{1}}
\left(y_1(1+y_2+\alpha y_1y_2 + y_1^2 y_2)^{-1}
\right)\\
&
-
 \tilde{L}
\left( y_1^{2}y_2
(1+2y_2 + y_2^2 + \alpha y_1y_2 + \alpha y_1 y_2^2 + y_1^2 y_2^2)^{-1}
\right)
\\
&
-
 \tilde{L}_{2,z_{1}}
\left(y_1y_2(1+y_2)^{-1}
\right)
-
 \tilde{L}
\left(y_2
\right)
=
0,
\end{split}
\end{align}
where $ \tilde{L}(x):=\tilde{L}(x)_{1,z_2}=L(x/(1+x))$ for the ordinary Rogers dilogarithm $L(x)$.

\subsection{Eight term relation for type $G_2$}
Let $n=2$ and  $d=(3,1)$.
We consider an initial seed $(B,x,y,z)$ with
\begin{align}
B=
\left(
\begin{matrix}
0 & -1\\
1 & 0
\end{matrix}
\right),
\quad z_1=(1,\alpha,\beta,1), 
\quad z_2=(1,1).
\end{align}
We choose $r_1=r_2=1$, so that $\tilde{r}_1=\tilde{r}_2=1$.
Then,
the following  sequence of alternative mutations is known to be periodic
\cite[Example 3.7]{Nakanishi15}:
\begin{align}
\label{eq:gseq4}
\begin{split}
(B,x,y,z)=(B[1],x[1],y[1],z[1])&
\buildrel  \mu_{1} \over \rightarrow
(B[2],x[2],y[2],z[2])
\buildrel  \mu_{2} \over \rightarrow
\\
&
\cdots
\buildrel  \mu_{2} \over \rightarrow
(B[9],x[9],y[9],z[9]).
\end{split}
\end{align}
The  data for the associated dilogarithm identity is given as follows:
\begin{align}
\begin{split}
&
y_1[1]=y_1, \quad y_2[2] = y_2(1+\alpha y_1 +\beta y_1^2 + y_1^3),\\
&
y_1[3]=y_1^{-1}(1+y_2+\alpha y_1y_2 +\beta y_1^2 y_2+y_1^3y_2),\\
&y_2[4] = y_1^{-3}y_2^{-1}
(1+3y_2 + 3y_2^2 + y_2^3 + 2 \alpha y_1y_2 + 4 \alpha y_1 y_2^2 + 2\alpha y_1 y_2^3\\
&\qquad\qquad +
\beta y_1^2 y_2 + \alpha^2 y_1^2 y_2^2 + 
3\beta y_1^2 y_2^2  + \alpha^2 y_1^2 y_2^3 + 2\beta y_1^2 y_2^3\\
&\qquad\qquad +
\alpha\beta y_1^3 y_2^2 + 2\alpha\beta y_1^3 y_2^3 
+ 3y_1^3 y_2^2 + 2y_1^3 y_2^3\\
&\qquad\qquad +
 \alpha y_1^4 y_2^2 + 2\alpha y_1^4 y_2^3
+ \beta^2 y_1^4 y_2^3 + 2\beta y_1^5 y_2^3 +
y_1^6 y_2^3
),
\\
&
y_1[5]=y_1^{-2}y_2^{-1}(1+2y_2+
y_2^2 + \alpha y_1 y_2 + \alpha y_1 y_2^2 + \beta y_1^2 y_2^2
+ y_1^3 y_2^2),
\\
&
y_2[6]=y_1^{-3}y_2^{-2}(
1 + 3y_2 + 3 y_2^2 + y_2^3
+\alpha y_1y_2 + 2\alpha y_1 y_2^2 +\alpha y_1 y_2^3\\
&\qquad\qquad +
\beta y_1^2 y_2^2 + \beta y_1^2 y_2^3 + y_1^3 y_2^3
),
\\
&
y_1[7]=y_1^{-1}y_2^{-1}(
1+y_2
),
\quad
y_2[8]=y_2^{-1},
\end{split}
\\
&
\varepsilon_1=\varepsilon_2=1,
\quad \varepsilon_3=\cdots =\varepsilon_8=-1,\\
&z_1[1]=z_1[5]=z_1,\quad z_1[3]=z_1[7]=z_1^*, \quad z_2[2] =z_2[4]=z_2[6]= 
z_2[8]=z_2.
\end{align}
Then, the dilogarithm identity in the form \eqref{eq:gid6} 
is explicitly written as follows:
\begin{align}
\label{eq:gid10}
\begin{split}
&
 \tilde{L}_{3,z_{1}}
\left(y_1
\right)
+
 \tilde{L}
\left(
 y_2(1+\alpha y_1 +\beta y_1^2 + y_1^3)
\right)\\
&
-
 \tilde{L}_{3,z_{1}}
\left(
y_1(1+y_2+\alpha y_1y_2 +\beta y_1^2 y_2+y_1^3y_2)^{-1}
\right)\\
&
-
 \tilde{L}
\bigl(
y_1^{3}y_2{}
(1+3y_2 + 3y_2^2 + y_2^3 + 2 \alpha y_1y_2 + 4 \alpha y_1 y_2^2 + 2\alpha y_1 y_2^3\\
&\qquad\qquad +
\beta y_1^2 y_2 + \alpha^2 y_1^2 y_2^2 + 
3\beta y_1^2 y_2^2  + \alpha^2 y_1^2 y_2^3 + 2\beta y_1^2 y_2^3\\
&\qquad\qquad +
\alpha\beta y_1^3 y_2^2 + 2\alpha\beta y_1^3 y_2^3 
+ 3y_1^3 y_2^2 + 2y_1^3 y_2^3\\
&\qquad\qquad +
 \alpha y_1^4 y_2^2 + 2\alpha y_1^4 y_2^3
+ \beta^2 y_1^4 y_2^3 + 2\beta y_1^5 y_2^3 +
y_1^6 y_2^3
)^{-1}
\bigr)
\\
&
-
 \tilde{L}_{3,z_{1}^*}
\bigl(
y_1^{2}y_2(1+2y_2+
y_2^2 + \alpha y_1 y_2 + \alpha y_1 y_2^2 + \beta y_1^2 y_2^2
+ y_1^3 y_2^2)^{-1}
\bigr)
\\
&
-
 \tilde{L}
\bigl(
y_1^{3}y_2^{2}(
1 + 3y_2 + 3 y_2^2 + y_2^3
+\alpha y_1y_2 + 2\alpha y_1 y_2^2 +\alpha y_1 y_2^3\\
&\qquad\qquad +
\beta y_1^2 y_2^2 + \beta y_1^2 y_2^3 + y_1^3 y_2^3)^{-1}
\bigr)
\\
&
-
 \tilde{L}_{3,z_{1}}
\left(
y_1{}y_2{}(
1+y_2
)^{-1}
\right)
-
 \tilde{L}
\left(
y_2
\right)
=
0.
\end{split}
\end{align}

\appendix

\section{Derivation of classical dilogarithm identity from  quantum one}

This appendix serves as  independent reading.

Here we complete the picture by showing how
the classical dilogarithm identity of higher degree 
in Theorem \ref{thm:gid3}
is obtained from   its quantum counterpart
in \cite[Theorem 4.1]{Nakanishi14b}.
This is a generalization of the argument  in \cite{Kashaev11}
for the ordinary cluster algebras with {\em skew-symmetric} exchange matrices.
(Thus, this presentation is new even 
for the ordinary cluster algebras with {\em skew-symmetrizable} exchange matrices.)

We rely on the {\em saddle point method},
which is standard  in quantum mechanics
(e.g., \cite[p.~95]{Takhtajan08}).
However, as in \cite{Kashaev11},
we stress that the derivation here is {\em only heuristic, and not functional-analytically rigorous}; for example, the uniqueness of the solution of the saddle point equations,
the specification of the integration contour through the saddle point, and the total validity of the method are not pursued.
Nevertheless, we believe that the derivation presented here is useful for  the readers.
(At least it is better than nothing.)

Here we follow and generalize the calculations  especially in  Section 4 and Appendix A of 
 \cite{Kashaev11}.
 Since this is a rather complicated subject,
we try to write it  in a self-contained way at a reasonable level, but not completely,
and we ask the readers to refer  to \cite{Kashaev11} (and also \cite{Nakanishi14b}) for further details.
 
 \subsection{Quantum dilogarithms of higher degree}
 For any positive integer $d$,
the {\em quantum dilogarithm of degree $d$ with coefficients
$z=(z_s)_{s=0}^d$}, where $z_0=z_d=1$, is defined as \cite{Nakanishi14b}
\begin{align}
\label{eq:gqd2}
\mathbf{\Psi}_{d,z,q}(x)
=
\prod_{k=0}^{\infty}
P_{d,z}(q^{2k+1} x)^{-1},
\end{align}
where $P_{d,z}(x)$ is the polynomial in
\eqref{eq:poly1},
whose coefficients $z_s$'s are  nonnegative real numbers 
which satisfy the generic condition
\eqref{eq:generic1},
and  $q\in \mathbb{C}$ with $|q|<1$.
Then, the power series \eqref{eq:gqd2} converges for
any  $x\in \mathbb{C}$.
Below let us concentrate on the region $x\geq 0$.
The function $\mathbf{\Psi}_{d,z,q}(x)$ is related to
the dilogarithm of higher degree
$\mathrm{Li}_{2;d,z}(x)$ in \eqref{eq:hLi2int1} 
in the asymptotic limit as follows \cite{Nakanishi14b}:
\begin{align}
\label{eq:gasym2}
\mathbf{\Psi}_{d,z,q}(x)
\sim
\exp\left(
-
\frac{\mathrm{Li}_{2;d,z}(-x)}{\log q^2}
\right),
\quad
q\rightarrow 1^{-}.
\end{align}

 \subsection{Quantum $Y$-seed and mutations}
 Let us recall the notions of a {\em quantum $Y$-seed} and its {\em mutation}
 in generalized cluster algebras following \cite{Nakanishi14b} with slight modification.
 \begin{defn}
 As in the classical  case,  first we fix the rank $n$ and and the mutation degree $d=(d_i)_{i=1}^n$.
 Then, we consider a triplet $(B,Y,z)$ such that
 \begin{itemize}
\item
$B=(b_{ij})_{i,j=1}^n$ is a skew-symmetrizable integer matrix  of size $n$,
\item
$Y=(Y_i)_{i=1}^n$ is an $n$-tuple  of {\em noncommutative} formal variables obeying the relation
\begin{align}
\label{eq:Y1}
Y_iY_j = q_j^{2b_{ji}} Y_j Y_i,
\quad
q_i :=q^{r_i},
\end{align}
where  $R=\mathrm{diag}(r_1,\dots,r_n)$ is a skew-symmetrizer of $B$,
\item
$z=(z_{i,s}\mid i=1,\dots,n; s=0,\dots,d_i)$ is a collection  of {\em commutative}
formal variables with $z_{i,0}=z_{i,d_i}=1$ for any $i=1,\dots,n$;
furthermore, $z_{i,s}$'s commute with  $Y_j$'s.
\end{itemize}
We call such $(B,Y,z)$ a {\em quantum $Y$-seed},
and call $Y=(Y_i)_{i=1}^n$ the {\em quantum $y$-variables} of $(B,Y,z)$.
\end{defn}

It is convenient to extend
the above quantum $y$-variables 
 to a family of noncommutative variables
$Y^{\alpha}$ ($\alpha\in \mathbb{Z}^n$) with the relations
\begin{align}
\label{eq:Y2}
q^{\langle\alpha,\beta\rangle}Y^{\alpha}Y^{\beta}
=Y^{\alpha+\beta},
\quad
\langle\alpha,\beta\rangle=:
\sum_{i,j=1}^n
\alpha_i r_i b_{ij}\beta_j,
\end{align}
where we identify $Y_i=Y^{e_i}$ for the $i$th unit vector $e_i$
of $\mathbb{Z}^n$.
 
 \begin{defn}
 For any quantum $Y$-seed $(B,Y,z)$ and any $k=1,\dots,n$, we define a new seed
$(B',Y',z')=\mu_k(B,Y,z)$, called the {\em mutation of $(B,Y,z)$ at $k$}, as follows:
\begin{align}
\label{eq:gYmut1}
Y'_i&=
\begin{cases}
Y_k^{-1}& i=k\\ 
Y^{e_i + d_k [b_{ki}]_+ e_k}
\displaystyle
\prod_{m=1}^{|b_{ki}|}
\left(
\sum_{s=0}^{d_k}
z_{k,s}
q_k^{-\mathrm{sgn}(b_{ki})
(2m-1)s}
Y_k^s
\right)
^{-\mathrm{sgn}(b_{ki})}
& i\neq k,
\end{cases}
\end{align}
where $\mathrm{sgn}(a)=1, -1, 0$ if $a>0$, $a<0$, $a=0$, respectively,
while $B'$ and $z'$ are defined 
by \eqref{eq:gBmut1} and \eqref{eq:gzmut1}, respectively.
\end{defn}
Indeed, it is easy to check that the following relation holds for the same 
skew-symmetrizer $R$:
\begin{align}
Y'_iY'_j = q_j^{2b'_{ji}} Y'_j Y'_i.
\end{align}
Again, the mutation $\mu_k$ is an involution,
i.e., $\mu_k(\mu_k(B,Y,z))=(B,Y,z)$.

\begin{rem} 
In \cite{Nakanishi14a} and \cite{Nakanishi14b}
a skew-symmetrizer $R$ was introduced, not for the exchange matrix $B$ itself,
but for the matrix $DB$.
Using this opportunity, let us
modify the convention to the one in this paper,
which is simpler.
For example, Equation (3.2) in \cite{Nakanishi14b} is replaced with
\eqref{eq:Y1};
and Equation (3.22) in \cite{Nakanishi14a} is replaced with
\begin{align}
R^{-1} (G^t)^T R C^t = I.
\end{align}

\end{rem}

 \subsection{Quantum dilogarithm identity of higher degree}
 
Let us choose any quantum $Y$-seed $(B,Y,z)$
as the {\em initial seed},
and consider a sequence of mutations starting from it:
\begin{align}
\label{eq:Gseq1}
\begin{split}
(B,Y,z)=(B[1],Y[1],z[1])&
\buildrel  \mu_{k_1} \over \rightarrow
(B[2],Y[2],z[2])
\buildrel  \mu_{k_2} \over \rightarrow
\\
& \cdots
\buildrel  \mu_{k_{m}} \over \rightarrow
(B[m+1],Y[m+1],z[m+1]),
\end{split}
\end{align}
where we use
 a common skew-symmetrizer
  $R=\mathrm{diag}(r_1,\dots,r_n)$ 
of $B[1]$, \dots, $B[m+1]$ to define the commutation relation
 for $Y[t]$  all $t=1,\dots,m+1$.

\begin{defn}
We say that the sequence \eqref{eq:Gseq1} is {\em $\sigma$-periodic}
for a permutation $\sigma$ of $\{1,\dots,n\}$ if
\begin{align}
\label{eq:Gperiod1}
\begin{split}
b_{\sigma(i)\sigma(j)}[m+1]
&=b_{ij},
\quad
Y_{\sigma(i)}[m+1]
=Y_{i},\quad
(i,j=1,\dots,n).
\end{split}
\end{align}
\end{defn}

Along with the sequence  \eqref{eq:Gseq1},
let us also consider the sequence 
\eqref{eq:gseq1} of the mutations of the classical seed, where the initial exchange matrix $B$ is taken to be
common in the both sequences.

\begin{conj}
The sequence \eqref{eq:Gseq1} is $\sigma$-periodic
if and only if the sequence \eqref{eq:gseq1}
is $\sigma$-periodic.
\end{conj}

The conjecture is known to be true for the ordinary case $d=(1,\dots,1)$
with skew-symmetric exchange matrices, i.e., $r=(1,\dots,1)$
(\cite[Proposition 3.4]{Kashaev11} and \cite[Theorem 5.1]{Inoue10a}).
However, we do not rely on this conjecture in the rest of the paper.

From now on we specialize the $z$-variables $z$
of the initial seed $(B,Y,z)$ in 
\eqref{eq:Gseq1} to be real positive  numbers such that,
for each $i=1,\dots, n$,
$z_i:=(z_{i,s})_{s=0}^{d_i}$ satisfies
the generic condition
\eqref{eq:generic1}.

\begin{thm} [{\cite[Theorem 4.1]{Nakanishi14b}}]
Suppose that the sequence \eqref{eq:Gseq1} is $\sigma$-periodic
for some permutation $\sigma$.
Then, the following equality holds:
\begin{align}
\label{eq:Gdilog1}
\mathbf{\Psi}_{d_{k_1},z_{k_1}[1],q_{k_1}}
(Y^{\varepsilon_1 c_{k_1}[1]})^{\varepsilon_1}
\cdots
\mathbf{\Psi}_{d_{k_m},z_{k_m}[m],q_{k_m}}
(Y^{\varepsilon_m c_{k_m}[m]})^{\varepsilon_m}
=1,
\end{align}
where $z_{k_t}[t]:=(z_{k_t,s})_{s=0}^{d_{k_t}}$ as before,
 $c_{k_t}[t]$  is the  $c$-vector  for the sequence \eqref{eq:gseq1}
 defined by
\eqref{eq:gtropy1},
and $\varepsilon_t$ 
is the tropical sign of $c_{k_t}[t]$ as in \eqref{eq:ts1}.
\end{thm}

\begin{rem}
The identity \eqref{eq:Gdilog1} is called the  quantum dilogarithm identity
{\em in tropical form} in \cite{Nakanishi14b} and proved only for the reciprocal case $z_i^*= z_i$
therein, where the $z$-variables do not mutate.
However, it is straightforwardly extended to the nonreciprocal case as above.
\end{rem}

 We will ``derive" the classical dilogarithm identity 
\eqref{thm:gid3} from \eqref{eq:Gdilog1} in the limit $q\rightarrow 1$.
Note that we cannot simply apply the formula \eqref{eq:gasym2} to 
\eqref{eq:Gdilog1},
since the quantum $y$-variables therein are not real numbers,
but noncommutative variables.
So, our strategy, which is standard in quantum mechanics,
is as follows:
\begin{itemize}
\item Step 1.
Represent those 
quantum  $y$-variables in \eqref{eq:Gdilog1}
 by operators acting on functions.
\item Step 2.
Take the expectation value of the left-hand side
of  \eqref{eq:Gdilog1},
and express it as an integral.
\item Step 3.
Evaluate the integral in the limit $q\rightarrow 1$ by the saddle point method.
\end{itemize}
This process ``magically"  transforms the (Euler) dilogarithms of
higher degree \eqref{eq:hLi2int1} into the Rogers dilogarithms of higher degree
\eqref{eq:Rogers10}.
See \eqref{eq:int3}--\eqref{eq:int4} for a  preview.


\subsection{Step 1: Operator representation}
We express the deformation parameter $q$ as
\begin{align}
\label{eq:qh1}
q= e^{\lambda^2 \hbar \sqrt{-1}},
\end{align}
where $\hbar$ is a positive real number, and $\lambda$ be a complex number
sufficiently close to $1$ such that
$\mathrm{Im}\, \lambda >0 $.
Later we will take $\hbar \rightarrow 0$ and $\lambda \rightarrow 1$.
 (See \cite[Appendix A.1]{Kashaev11} for the explanation of  introducing the parameter $\lambda$.)
 
Consider the Hilbert space $L^2(\mathbb{R}^n)$. Let $u=(u_i)_{i=1}^n$ denote the  coordinate
of $\mathbb{R}^n$.
Let $\hat{u}_i$ and $\hat{p}_i$ be the standard {\em position} and {\em momentum operators} on 
$L^2(\mathbb{R}^n)$ (densely) defined by
\begin{align}
(\hat{u}_i f)(u) = u_i f(u),
\quad
(\hat{p}_i f)(u) = \frac{\hbar}{\sqrt{-1}} \frac{\partial f}{\partial u_i}(u),
\quad
f\in L^2(\mathbb{R}^2).
\end{align}
Thus, we have the commutation relations
\begin{align}
[\hat{u}_i,\hat{u}_j]=
[\hat{p}_i, \hat{p}_j]=0,
\quad
[\hat{p}_i, \hat{u}_j]=\frac{\hbar}{\sqrt{-1} }\delta_{ij}.
\end{align}

Let $B$ be the initial exchange matrix of the sequence 
\eqref{eq:Gseq1}. Define
\begin{align}
\label{eq:op1}
\hat{w}_i = \sum_{j=1}^n b_{ji} \hat{u}_j,
\quad
\hat{D}_i = r_i \hat{p}_i + \hat{w}_i,
\quad
\hat{{Y}}_i
=
\exp 
(\lambda \hat{D}_i ).
\end{align}
Then, we have
\begin{align}
\label{eq:com1}
[\hat{D}_i, \hat{D}_j]
=2\hbar \sqrt{-1} r_j b_{ji},
\quad
\hat{{Y}}_i
\hat{{Y}}_j
=q_j^{2b_{ji}} 
\hat{{Y}}_j
\hat{{Y}}_i.
\end{align}
Thus, we have a representation of the initial quantum $y$-variables
satisfying \eqref{eq:Y1}.
More generally, for any $\alpha\in \mathbb{Z}^n$, we define
\begin{align}
\label{eq:Yhat2}
\hat{{Y}}^{\alpha}=
\exp (\lambda\alpha \hat{D}),
\quad
\alpha \hat{D}
=\sum_{i=1}^n
\alpha_i \hat{D}_i.
\end{align}
Then, 
\begin{align}
\label{eq:Y21}
q^{\langle\alpha,\beta\rangle}\hat{Y}^{\alpha}\hat{Y}^{\beta}
=\hat{Y}^{\alpha+\beta},
\quad
\langle\alpha,\beta\rangle=
\sum_{i,j=1}^n
\alpha_i r_i b_{ij}\beta_j,
\end{align}
Thus, they give a representation of $Y^{\alpha}$'s in \eqref{eq:Y2}.

Next we describe the mutations of the quantum $y$-variables.
Along with the sequence 
\eqref{eq:Gseq1} with $\sigma$-periodicity,
we introduce a sequence of linear  transformations
\begin{align}
\label{eq:seqmap1}
\mathbb{R}^n
\buildrel \rho_1\over \rightarrow
\mathbb{R}^n
\buildrel \rho_2\over \rightarrow
\cdots
\buildrel \rho_m\over \rightarrow
\mathbb{R}^n
\buildrel \sigma \over \rightarrow
\mathbb{R}^n,
\end{align}
where, for $t=1,\dots,m$,
\begin{align}
\label{eq:rho1}
\begin{split}
\rho_t &:  \mathbb{R}^n  \rightarrow  \mathbb{R}^n,
\quad (u_i)_{i=1}^n \mapsto (u'_i)_{i=1}^n,\\
u'_i&=
\begin{cases}
\displaystyle
-u_{k_t} +
d_{k_t}
\sum_{j=1}^n [-\varepsilon_t b_{jk_t}[t]]_+  u_j
& i = k_t
\\
u_i & i \neq k_t,
\end{cases}
\end{split}
\end{align}
and
\begin{align}
\sigma &:  \mathbb{R}^n  \rightarrow  \mathbb{R}^n,
\quad (u_i)_{i=1}^n \mapsto (u'_i)_{i=1}^n,
\quad
u'_i=u_{\sigma(i)}.
\end{align}

The sequence \eqref{eq:seqmap1} induces the sequence of the maps
\begin{align}
\label{eq:seqmap2}
L^2(\mathbb{R}^n)
\buildrel \rho_1^*\over \leftarrow
L^2(\mathbb{R}^n)
\buildrel \rho_2^*\over \leftarrow
\cdots
\buildrel \rho_m^*\over \leftarrow
L^2(\mathbb{R}^n)
\buildrel \sigma^* \over \leftarrow
L^2(\mathbb{R}^n),
\end{align}
where $\rho_t^* (f) = f\circ \rho_t$ and $\sigma^* (f) = f\circ \sigma$ for $f\in L^2(\mathbb{R}^n)$.

\begin{lem} 
If the sequence \eqref{eq:Gseq1} is $\sigma$-periodic,
the following periodicity holds:
\begin{align}
\label{eq:rho4}
\sigma \circ \rho_m \circ \dots \circ \rho_1 &= \mathrm{id},
\\
\label{eq:rho3}
\rho_1^* \circ \cdots \circ \rho_m^* \circ \sigma^*& = \mathrm{id}.
\end{align}
\end{lem}

\begin{proof}
Suppose that the sequence \eqref{eq:Gseq1} is $\sigma$-periodic.
Then, taking $q\rightarrow 1$ limit, the corresponding (classical)
$y$-variables in the sequence  \eqref{eq:gseq1} are also $\sigma$-periodic.
Then, the associated $c$-vectors in \eqref{eq:gtropy1} are also $\sigma$-periodic.
Thus, the associated $g$-vectors \cite{Nakanishi14a}, which are not explained here,
are also $\sigma$-periodic due to the duality of $c$- and $g$-vectors
\cite[Proposition 3.21]{Nakanishi14a}.
On the other hand, the linear transformations $\rho_t$ and $\sigma$
exactly describe the mutations of $g$-vectors along the sequence 
 \eqref{eq:gseq1}. Therefore, we have the equality \eqref{eq:rho4}.
The equality \eqref{eq:rho3} follows from \eqref{eq:rho4}.
\end{proof}

For any invertible linear map $\Upsilon: L^2(\mathbb{R}^n) \rightarrow L^2(\mathbb{R}^n)$ 
and
any linear operator $\hat{O}$ acting on $L^2(\mathbb{R}^n)$, let
\begin{align}
\mathsf{Ad}(\Upsilon)(\hat{O}):=
\Upsilon \circ \hat{O} \circ \Upsilon^{-1}.
\end{align}

\begin{lem}
\label{lem:ad2}
The following formulas hold:
For $t=1,\dots,m$,
\begin{align}
\label{eq:ad1}
\mathsf{Ad}(\rho_t^*)(\hat{u}_i) &=
\begin{cases}
\displaystyle
-\hat{u}_{k_t} +
d_{k_t}\sum_{j=1}^n [-\varepsilon_t b_{jk_t}[t]]_+  \hat{u}_j
& i = k_t
\\
\hat{u}_i
& i\neq k_t,
\end{cases}
\\
\label{eq:ad2}
\mathsf{Ad}(\rho_t^*)(\hat{p}_i) &=
\begin{cases}
-\hat{p}_{k_t}
& i = k_t
\\
\hat{p}_i
+d_{k_t}[-\varepsilon b_{ik_t}[t]]_+\hat{p}_{k_t}
& i\neq k_t,
\end{cases}
\\
\label{eq:ad3}
\mathsf{Ad}(\sigma^*)(\hat{u}_i) &=\hat{u}_{\sigma(i)},
\quad
\mathsf{Ad}(\sigma^*)(\hat{p}_i) =\hat{p}_{\sigma(i)}.
\end{align}
\end{lem}
\begin{proof}
The equalities \eqref{eq:ad1} and \eqref{eq:ad3} are immediate from the definitions
of $\rho_t$ and $\sigma$.
The equality \eqref{eq:ad2} is obtained
by applying the chain rule for
the inverse of $\rho_t$
\begin{align}
\label{eq:rho2}
\begin{split}
(\rho_t)^{-1} &:  \mathbb{R}^n  \rightarrow  \mathbb{R}^n,
\quad (u'_i)_{i=1}^n \mapsto (u_i)_{i=1}^n,\\
u_i&=
\begin{cases}
\displaystyle
-u'_{k_t} +
d_{k_t}
\sum_{j=1}^n [-\varepsilon_t b_{jk_t}[t]]_+  u'_j
& i = k_t
\\
u'_i & i \neq k_t.
\end{cases}
\end{split}
\end{align}
\end{proof}

\begin{rem}
In \cite{Kashaev11}
the separate symbols  $u[t]$  ($t=1,\dots,m+2$) 
are employed for the coordinate
of each space $\mathbb{R}^n$
 in the sequence \eqref{eq:seqmap2} from left to right.
Here, we do not use them for  simplicity.
\end{rem}

Let us define the following operators for $t=1,\dots, m$. (For $t=1$, it is already given
in \eqref{eq:op1} and \eqref{eq:Yhat2}.)
\begin{align}
\label{eq:op2}
\hat{w}_i[t] &= \sum_{j=1}^n b_{ji}[t] \hat{u}_j,
\quad
\hat{D}_i[t] = r_i \hat{p}_i + \hat{w}_i[t],
\quad
\hat{{Y}}_i[t]
=
\exp 
(\lambda \hat{D}_i[t] ),
\\
\label{eq:Yhat3}
\hat{{Y}}^{\alpha}[t]&=
\exp (\lambda\alpha \hat{D}[t]),
\quad
\alpha \hat{D}[t]
=\sum_{i=1}^n
\alpha_i \hat{D}_i[t].
\end{align}
Then, like the $t=1$ case  \eqref{eq:com1},
we have the following commutation relations: For $t=1,\dots,m$,
\begin{align}
\label{eq:com2}
[\hat{D}_i[t], \hat{D}_j[t]]
&=2\hbar \sqrt{-1} r_j b_{ji}[t],
\quad
\hat{{Y}}_i[t]
\hat{{Y}}_j[t]
=q_j^{2b_{ji}[t]} 
\hat{{Y}}_j[t]
\hat{{Y}}_i[t],
\\
\label{eq:Y3}
q^{\langle\alpha,\beta\rangle_t}Y^{\alpha}[t]Y^{\beta}[t]
&=Y^{\alpha+\beta}[t],
\quad
\langle\alpha,\beta\rangle_t=:
\sum_{i,j=1}^n
\alpha_i r_i b_{ij}[t]\beta_j.
\end{align}

\begin{lem}
\label{lem:ad1}
The following formulas hold: For $t=1,\dots,m$,
\begin{align}
\label{eq:ad4}
\mathsf{Ad}(\rho_t^*)(\hat{w}_i[t+1]) &=
\begin{cases}
-\hat{w}_{k_t}[t] 
& i = k_t
\\
\hat{w}_i[t]
+d_{k_t}[\varepsilon_t b_{k_ti}[t]]_+\hat{w}_{k_t}[t]
& i\neq k_t,
\end{cases}
\\
\label{eq:ad5}
\mathsf{Ad}(\rho_t^*)(r_i\hat{p}_i) &=
\begin{cases}
-r_{k_t}\hat{p}_{k_t}
& i = k_t
\\
r_i\hat{p}_i
+d_{k_t}[\varepsilon_t b_{k_t i}[t]]_+ r_{k_t}\hat{p}_{k_t}
& i\neq k_t.
\end{cases}
\end{align}
\end{lem}
\begin{proof}
They follow from Lemma \ref{lem:ad2}.
\end{proof}

\begin{rem}
A subtle difference between the second formulas of
\eqref{eq:ad2} and \eqref{eq:ad5} is important,
since $b_{k_t i}[t]\neq -b_{i k_t}[t]$ in general.
\end{rem}

The conclusion of this section is given as follows:
\begin{prop}
\label{prop:Yrep1}
(1)  Quantum tropical mutations: For $t=1,\dots,m$,
\begin{align}
\label{eq:gYhatmut1}
\mathsf{Ad}(\rho_t^*)(\hat{Y}_i[t+1])&=
\begin{cases}
\hat{Y}_{k_t}[t]^{-1}& i=k_t\\ 
\hat{Y}^{e_i + d_{k_t} [\varepsilon_t b_{k_t i}[t]]_+ e_{k_t}}[t]
& i\neq k_t.
\end{cases}
\end{align}
\par
(2) 
For $t=1,\dots,m$,
\begin{align}
\label{eq:ad9}
\mathsf{Ad}(\rho_1^*)\cdots\mathsf{Ad}(\rho_t^*)(\hat{Y}_{i}[t+1])
=\hat{Y}^{c_i[t+1]}.
\end{align}
\end{prop}
\begin{proof}
(1). 
This follows from Lemma \ref{lem:ad1}.
(2). Note that the second formula of 
\eqref{eq:crec1} is also written as
\begin{align}
\label{eq:crec2}
c_{ij}[t+1]=
c_{ij}[t]
+
d_{k_t}(
[c_{ik_t}[t]]_+ b_{k_t j}[t]
+ c_{ik_t }[t]
[-b_{k_t j}[t]]_+)
\quad
j\neq k_t.
\end{align}
In particular, for the tropical sign $\varepsilon_t$,
we have $[-\varepsilon_t c_{ik_t}[t]]_+=0$.
Thus, it is also equivalent to
\begin{align}
\label{eq:crec3}
c_{ij}[t+1]=
c_{ij}[t]
+
d_{k_t} c_{ik_t}[t]
[\varepsilon_t b_{k_tj}[t]]_+
\quad
j\neq k_t.
\end{align}
Comparing \eqref{eq:gYhatmut1} and \eqref{eq:crec3},
we inductively obtain the formula \eqref{eq:ad9} for $t=1,\dots,m$.
\end{proof}

\subsection{Step 2: Integral expression}
In the quantum dilogarithm identity \eqref{eq:Gdilog1}
let us replace the initial quantum  $y$-variables $Y_i$
with their operator representations $\hat{Y}_i$.
Then, multiplying the left-hand side of the equality \eqref{eq:rho3},
we obtain the following equality:
\begin{align}
\begin{split}
\label{eq:Gdilog2}
\mathbf{\Psi}_{d_{k_1},z_{k_1}[1],q_{k_1}}
(\hat{Y}^{\varepsilon_1 c_{k_1}[1]})^{\varepsilon_1}
&
\cdots
\mathbf{\Psi}_{d_{k_m},z_{k_m}[m],q_{k_m}}
(\hat{Y}^{\varepsilon_m c_{k_m}[m]})^{\varepsilon_m}
\rho_1^*  \cdots  \rho_m^*  \sigma^* = \mathrm{id},
\end{split}
\end{align}
where the composition symbol $\circ$ is omitted for simplicity.

Using \eqref{eq:ad9} repeatedly,
it is transformed into the following equality,
which is a generalization of the quantum dilogarithm identity {\em in  local form}
in \cite{Kashaev11}:
\begin{align}
\begin{split}
\label{eq:Gdilog3}
\mathbf{\Psi}_{d_{k_1},z_{k_1}[1],q_{k_1}}
(\hat{Y}_{k_1}[1]^{\varepsilon_1})^{\varepsilon_1}
\rho_1^* 
& \mathbf{\Psi}_{d_{k_2},z_{k_2}[2],q_{k_2}}
(\hat{Y}_{k_2}[2]^{\varepsilon_2})^{\varepsilon_2}
\rho_2^* \\
&
\cdots
\mathbf{\Psi}_{d_{k_m},z_{k_m}[m],q_{k_m}}
(\hat{Y}_{k_m}[m]^{\varepsilon_m})^{\varepsilon_m}
   \rho_m^*  \sigma^* = \mathrm{id}.
\end{split}
\end{align}

Using the Dirac's bra-ket notation, 
 we introduce a family of  common eigenvectors
$ | u \rangle$ and $| p \rangle$
($ u, p \in \mathbb{R}^n$)
 of
the position  operators $\hat{u}_i $'s and momentum operators $\hat{p}_i$'s,
respectively, 
and their complex conjugate $\langle u |$ and $\langle p |$.
They satisfy
\begin{align}
\hat{u}_i | u \rangle &= u_i | u \rangle,
\quad
\hat{p}_i | p \rangle = p_i | p \rangle,
\\
\langle u | u' \rangle &= \prod_{i=1}^n \delta(u_i-u'_i),
\quad
\langle p | p' \rangle = (2\pi \hbar)^n \prod_{i=1}^n \delta(p_i-p'_i),
\\
\langle u | p \rangle &= \exp \left( \frac{\sqrt{-1}}{\hbar} u p\right),
\quad
\langle p | u \rangle = \exp \left( - \frac{\sqrt{-1}}{\hbar} u p\right),
\quad
u p= \sum_{i=1}^n u_i p_i.
\end{align} 
In particular, we have,
for vectors $\langle u | $ and $ |p \rangle$,
\begin{align}
\label{eq:exp1}
\frac{\langle u |  \hat{D}_i[t] |p \rangle}
{\langle u | p \rangle}
=
 r_i p_i + w_i[t],
\quad
w_i[t]=
\sum_{j=1}^n b_{ji}[t] u_j.
\end{align}
We  also have the completeness property,
\begin{align}
1= \int du   | u \rangle \langle u|,
\quad
1= \int \frac{dp}{(2\pi\hbar)^n}   | p \rangle \langle p|.
\end{align}

Let
\begin{align}
\hat{O}=
\label{eq:Gdilog4}
\mathbf{\Psi}_{d_{k_1},z_{k_1}[1],q_{k_1}}
(\hat{Y}_{k_1}[1])^{\varepsilon_1}
\rho_1^* 
\cdots
\mathbf{\Psi}_{d_{k_m},z_{k_m}[m],q_{k_m}}
(\hat{Y}_{k_m}[m])^{\varepsilon_m}
   \rho_m^*  \sigma^* 
\end{align}
be the left-hand side of \eqref{eq:Gdilog3},
{\em which is actually the identity operator} due to 
\eqref{eq:Gdilog3}.
Choose an arbitrary position eigenvector $|u[1]\rangle$.
Then, we introduce the momentum eigenvector $|\tilde{p}[1]\rangle$
whose eigenvalues are given by
\begin{align}
\tilde{p}_i[1]=w_i[1] :=\sum_{i=1}^n b_{ji}[1] u_j[1].
\end{align}
Let
\begin{align}
F(u[1],\tilde{p}[1]):=\frac{\langle u[1] |  \hat{O} |\tilde{p}[1] \rangle}
{\langle u[1] | \tilde{p}[1] \rangle},
\end{align}
{\em which is actually $1$}.
Skipping some detail (see   \cite[Sections 5.2 and A.3]{Kashaev11}),
we obtain the following integral expression, using
\eqref{eq:exp1}:
\begin{align}
\label{eq:int1}
\begin{split}
F(u[1],\tilde{p}[1])
=
&(2\pi \hbar)^{-n(m-1)}
\int dp[1] \cdots dp[m-1] du[2] \dots du[m]\\
&\times 
\mathbf{\Psi}_{d_{k_1},z_{k_1}[1],q_{k_1}}
(y_{k_1}[1]^{\varepsilon_1})^{\varepsilon_1}
 \exp \left( \frac{\sqrt{-1}}{\hbar} u[1] (p[1]-\tilde{p}[1])\right)
 \cdots
 \\
 &\times 
\mathbf{\Psi}_{d_{k_m},z_{k_m}[m],q_{k_m}}
(y_{k_m}[m]^{\varepsilon_m})^{\varepsilon_m}
 \exp \left( \frac{\sqrt{-1}}{\hbar} u[m] (p[m]-\tilde{p}[m])\right),
\end{split}
\end{align}
where ${p}_i[m]$ is determined by
\begin{align}
\label{eq:rp2}
r_{\sigma^{-1}(i)} \tilde{p}_{\sigma^{-1}(i)}[1]
&=
\begin{cases}
-r_{k_m} {p}_{k_m}[m] 
& i = k_m
\\
r_{i}{p}_{i}[m]
+d_{k_m}[\varepsilon_m b_{k_m i}[m]]_+ r_{k_m}{p}_{k_m}[m]
& i\neq k_m,
\end{cases}
\end{align}
while
$\tilde{p}_i[t]$ ($t=2,\dots,m$)
and $y_{k_t}[t]$ ($t=1,\dots,m$) are dependent variables 
of the integration variables such that,
for $t=1,\dots,m-1$,
\begin{align}
\label{eq:rp1}
r_i \tilde{p}_i[t+1]
&=
\begin{cases}
-r_{k_t} {p}_{k_t}[t] 
& i = k_t
\\
r_i{p}_i[t]
+d_{k_t}[\varepsilon_t b_{k_t i}[t]]_+ r_{k_t}{p}_{k_t}[t]
& i\neq k_t,
\end{cases}
\end{align}
and, for $t=1,\dots,m$,
\begin{align}
\label{eq:ypw1}
y_{k_t}[t]&:=\exp \left( \lambda( r_{k_t} p_{k_t}[t] + w_{k_t}[t])\right),
\quad
w_{k_t}[t]=
\sum_{i=1}^n b_{jk_t}[t] u_j[t].
\end{align}

\begin{rem}
The symmetry of the skew-symmetrizer \eqref{eq:rsym1} is used to 
obtain the relation \eqref{eq:rp2}.
Without it, the left-hand side of \eqref{eq:rp2}
is $r_i \tilde{p}_{\sigma^{-1}(i)}$, which we do not want.
(Compare with the forthcoming \eqref{eq:se6}.)
\end{rem}

 \subsection{Step 3: Saddle point method}

In our parametrization of $q$ in \eqref{eq:qh1},
the asymptotic behavior of the quantum dilogarithms of higher degree
in \eqref{eq:gasym2}  is expressed as
\begin{align}
\label{eq:gasym3}
\mathbf{\Psi}_{d,z,q_i}(x)
\sim
\exp\left(
\frac{\sqrt{-1}}{\hbar}
\frac{1}{2\lambda^2 r_i}
\mathrm{Li}_{2;d,z}(-x)
\right),
\quad
\hbar\rightarrow 0^+.
\end{align}

We are interested in the asymptotic behavior of \eqref{eq:int1}
in the limit $\hbar \rightarrow 0^+$.
Thus, we may replace the quantum dilogarithms
therein by the right-hand side of \eqref{eq:gasym3},
and we have
\begin{align}
\label{eq:int2}
\begin{split}
F(u[1],\tilde{p}[1])
\sim
&(2\pi \hbar)^{-n(m-1)}
\int dp[1] \cdots dp[m-1] du[2] \dots du[m]\\
 &\times 
 \exp \left( \frac{\sqrt{-1}}{\hbar} 
 \sum_{t=1}^m 
 \left\{
 \frac{{\varepsilon_t}}{2\lambda^2 r_{k_t}}
\mathrm{Li}_{2;d_{k_t},z_{k_t}[t]}(-y_{k_t}[t]^{\varepsilon_t})
+
 u[t] (p[t]-\tilde{p}[t])
 \right\}\right).
\end{split}
\end{align}

\begin{rem}
Due to our assumption of $\lambda \approx 1$ with $\mathrm{Im}\, \lambda>0$,
$y_{k_t}[t]$ defined by \eqref{eq:ypw1} is a complex number close to a positive real number.
Accordingly, the functions $\mathrm{Li}_{2;d_{k_t},z_{k_t}[t]}(x)$ 
in \eqref{eq:int2} are analytically continued
in the vicinity of the negative real line $\mathbb{R}_-$.
\end{rem}

To evaluate the integral \eqref{eq:int2} in the limit $\hbar \rightarrow 0^+$,
we apply the {\em saddle point method}.

To do that, we need the following formulas.
\begin{lem} 
\begin{align}
\label{eq:der2}
r_i \frac{\partial}{\partial u_i[t]}
\left(
 \frac{{\varepsilon_t}}{2\lambda^2 r_{k_t}}
\mathrm{Li}_{2;d_{k_t},z_{k_t}[t]}(-y_{k_t}[t]^{\varepsilon_t})
\right)
&=
-\frac{1}{\lambda}\log
\left( 
P_{d_{k_t},z_{k_t}[t]}(y_{k_t}[t]^{\varepsilon_t})
\right)^{-b_{k_t i}[t]/2} ,
\\
\label{eq:der3}
\frac{\partial}{\partial p_i[t]}
\left(
 \frac{{\varepsilon_t}}{2\lambda^2 r_{k_t}}
\mathrm{Li}_{2;d_{k_t},z_{k_t}[t]}(-y_{k_t}[t]^{\varepsilon_t})
\right)
&=
\delta_{i k_t}
\frac{1}{\lambda}\log
\left( 
P_{d_{k_t},z_{k_t}[t]}(y_{k_t}[t]^{\varepsilon_t})
\right)^{-1/2} .
\end{align}
\end{lem}
\begin{proof}
These are obtained by \eqref{eq:hLi2int1}, 
\eqref{eq:ypw1}, and the skew-symmetric property of $RB[t]$.
\end{proof}

The {\em saddle point equation} of the integral \eqref{eq:int2}
is the extremum condition of its integrand with respect to
the integral variables $p[t]$ ($t=1,\dots,m-1$) and 
$u[t]$ ($t=2,\dots,m$).

\par
(a). {\em Extremum condition for $u_i(t)$} ($t=2,\dots,m$).
By differentiating the integrand of \eqref{eq:int2} by
$u_i(t)$ using \eqref{eq:der2}, and multiplying $r_i$, we have
\begin{align}
\label{eq:se1}
-\frac{1}{\lambda}\log
\left( 
P_{d_{k_t},z_{k_t}[t]}(y_{k_t}[t]^{\varepsilon_t})
\right)^{-b_{k_t i}[t]/2}+r_ip_i[t]- r_i \tilde{p}_i[t]=0.
\end{align}
Note that, in particular,
\begin{align}
\label{eq:se11}
p_{k_t}[t]=\tilde{p}_{k_t}[t].
\end{align}
By \eqref{eq:se1} and \eqref{eq:rp1}, we also have, for $t=2,\dots,m-1$,
\begin{align}
\label{eq:se2}
e^{\lambda r_i \tilde{p}_i[t+1]}
=
\begin{cases}
(e^{\lambda r_{k_t} \tilde{p}_{k_t}[t]})^{-1}
& i=k_t
\\
e^{\lambda r_{i} \tilde{p}_i[t]}
\left(
(e^{\lambda r_{k_t} \tilde{p}_{k_t}[t]})^{[\varepsilon_t b_{k_t i}[t]]_+}
\right)^{d_{k_t} }
\\
\quad \times
\left( 
P_{d_{k_t},z_{k_t}[t]}(y_{k_t}[t]^{\varepsilon_t})
\right)^{-b_{k_t i}[t]/2}
& i \neq k_t.
\end{cases}
\end{align}

\par
(b). {\em Extremum condition for $p_i(t)$} ($t=1,\dots,m-1$).
By differentiating the integrand of \eqref{eq:int2} by
$p_i(t)$ using \eqref{eq:der3} and \eqref{eq:rp1},
we have,
for $i=k_t$,
\begin{align}
\label{eq:se3}
\begin{split}
& \frac{1}{\lambda}\log
\left( 
P_{d_{k_t},z_{k_t}[t]}(y_{k_t}[t]^{\varepsilon_t})
\right)^{-1/2}+u_{k_t}[t]\\
& \qquad
-\sum_{j=1}^n
d_{k_t} [\varepsilon_t b_{k_t i}[t]]_+
 u_j[t+1]
+ u_{k_t}[t+1]
=0,
\end{split}
\end{align}
and otherwise,
\begin{align}
\label{eq:se5}
u_i[t]-u_i[t+1]=0,
\quad i\neq t_k.
\end{align}
By \eqref{eq:se3} and \eqref{eq:gBmut1} (
noticing that
$[-b_{ik}]_+b_{kj} + b_{ik}[b_{kj}]_+=
[b_{ik}]_+b_{kj} + b_{ik}[-b_{kj}]_+$
),
we also have the following equations for $w_i[t]=
\sum_{j=1}^n b_{ji}[t]u_j[t]$ for $t=1,\dots,m-1$:
\begin{align}
\label{eq:se4}
e^{\lambda w_i[t+1]}
=
\begin{cases}
(e^{\lambda w_{k_t}[t]})^{-1}
& i=k_t
\\
e^{\lambda w_i[t]}
\left(
(e^{\lambda w_{k_t}[t]})^{[\varepsilon_t b_{k_t i}[t]]_+}
\right)^{d_{k_t} }
\\
\quad \times
\left( 
P_{d_{k_t},z_{k_t}[t]}(y_{k_t}[t]^{\varepsilon_t})
\right)^{-b_{k_t i}[t]/2}
& i \neq k_t.
\end{cases}
\end{align}

A complex (but almost positive real) solution $u_i[t]$ ($t=2,\dots,m$),
$p_i[t]$ ($t=1,\dots,m-1$) of \eqref{eq:se1} and \eqref{eq:se3}
is constructed as follows.
\begin{itemize}
\item[(1).] ($y$-variables) We have $u_i[1]$ as initial data,
from which $w_i[1]$ is uniquely determined.
We set $y_i[1]=\exp(2\lambda w_i[1])$.
Then, $y_i[t]$ ($t=1,\dots,m$) are determined by the mutation sequence
\eqref{eq:gseq1}.
\item[(2).] ($u$-variables) We determine $u_i[t]$ ($t=2,\dots,m$) by
the extremum condition \eqref{eq:se3} and \eqref{eq:se5}.
\item[(3).] ($p$-variables) Set $\tilde{p}_i[t]$ ($t=2,\dots,m$)
by $\exp(\lambda r_i\tilde{p}_i[t])=y_i[t]^{1/2}$.
Then, determine $p_i[t]$ ($t=1,\dots,m-1$) by \eqref{eq:rp1},
and $p_i[m]$ by \eqref{eq:rp2}.
\end{itemize}
It is necessary to check that \eqref{eq:ypw1} and \eqref{eq:se1} are satisfied.
\par
 The condition \eqref{eq:ypw1}: By \eqref{eq:se3} and \eqref{eq:se5},
we have \eqref{eq:se4}.
Thus, we have $\exp(\lambda r_i\tilde{p}_i[t])=\exp(\lambda w_i[t])=y_i[t]^{1/2}$.
Therefore, the condition \eqref{eq:ypw1} is satisfied.
\par
The condition \eqref{eq:se1}:
Since $\exp(\lambda r_i\tilde{p}_i[t])=y_i[t]^{1/2}$,
the condition \eqref{eq:se2} is satisfied.
Then, combining it with \eqref{eq:rp1},
we obtain  \eqref{eq:se1} for $t=1,\dots,m-1$.
To obtain \eqref{eq:se1} for $t=m$ requires a little extra consideration.
Extend the above construction for $y_i[m+1]$ and $\tilde{p}_i[m+1]$.
Then, the condition \eqref{eq:se2} is satisfied for $t=m$.
Moreover, thanks to the $\sigma$-periodicity of $y$-variables,
we have  
\begin{align}
r_{\sigma(i)}\tilde{p}_{\sigma(i)}[m+1]=r_i \tilde{p}_i[1].
\end{align}
Thus, we have
\begin{align}
\label{eq:se6}
e^{\lambda r_{\sigma^{-1}(i)} \tilde{p}_{\sigma^{-1}(i)}[1]}
=
\begin{cases}
(e^{\lambda r_{k_m} \tilde{p}_{k_m}[m]})^{-1}
& i=k_m
\\
e^{\lambda r_{i} \tilde{p}_i[m]}
\left(
(e^{\lambda r_{k_m} \tilde{p}_{k_m}[m]})^{[\varepsilon_t b_{k_m i}[m]]_+}
\right)^{d_{k_m} }
\\
\quad \times
\left( 
P_{d_{k_m},z_{k_m}[m]}(y_{k_m}[m]^{\varepsilon_m})
\right)^{-b_{k_m i}[m]/2}
& i \neq k_m.
\end{cases}
\end{align}
Then, combining it with \eqref{eq:rp2},
we obtain  \eqref{eq:se1} for $t=m$.

\par
Thus, this is indeed a solution of the saddle point equation.

The saddle point method claims that, under ``some" condition
which we do not pursue in this paper,
 the integral \eqref{eq:int2} in the limit $\hbar \rightarrow 0$
is approximated at the value of the integrand
at an extremum point, up to some multiplicative factor which is irrelevant here.
(See \cite[p. 95]{Takhtajan08} for the explicit expression for the one variable case.)
Therefore, taking the above solution,  ignoring the multiplicative factor,
then taking the logarithm and removing the factor $\sqrt{-1}/\hbar$, we have
\begin{align}
\label{eq:int3}
 \sum_{t=1}^m 
 \left\{
 \frac{{\varepsilon_t}}{2\lambda^2 r_{k_t}}
\mathrm{Li}_{2;d_{k_t},z_{k_t}[t]}(-y_{k_t}[t]^{\varepsilon_t})
+
 u[t] (p[t]-\tilde{p}[t])
 \right\}.
\end{align}
Recall that, for our solution,
\begin{align}
p_i[t]-\tilde{p}_i[t]&=\frac{1}{\lambda r_i}
\log 
\left( 
P_{d_{k_t},z_{k_t}[t]}(y_{k_t}[t]^{\varepsilon_t})
\right)^{-b_{k_t i}[t]/2},
\\
w_i[t]&=\frac{1}{2\lambda} \log y_i[t].
\end{align}
Thus, using the skew-symmetric property
$
b_{k_t i}[t] r_i^{-1} = - b_{ik_t}[t]r_{k_t}^{-1}
$,
we have
\begin{align}
\begin{split}
\sum_{i=1}^n 
 u_i[t] (p_i[t]-\tilde{p}_i[t])
 &=
 \frac{1}{\lambda }
 \sum_{i=1}^n 
 \frac{1}{r_i }
 \left(-\frac{b_{k_t i}[t]}{2}
 \right)
  u_i[t] 
  \log 
P_{d_{k_t},z_{k_t}[t]}(y_{k_t}[t]^{\varepsilon_t})
\\
&=
 \frac{1}{2 \lambda r_{k_t}}
 \sum_{i=1}^n 
 b_{i k_t }[t]
  u_i[t] 
  \log 
P_{d_{k_t},z_{k_t}[t]}(y_{k_t}[t]^{\varepsilon_t})
\\
&=
 \frac{1}{2 \lambda r_{k_t}}
  w_{k_t}[t] 
  \log 
P_{d_{k_t},z_{k_t}[t]}(y_{k_t}[t]^{\varepsilon_t})
\\
&=
 \frac{\varepsilon_t}{4 \lambda^2 r_{k_t}}
 \log y_i[t]^{\varepsilon_t}
 \cdot
  \log 
P_{d_{k_t},z_{k_t}[t]}(y_{k_t}[t]^{\varepsilon_t}).
\\
\end{split}
\end{align}
Thus, the expression \eqref{eq:int3} is  equal to
\begin{align}
\label{eq:int4}
 \frac{-1}{2\lambda^2}
 \sum_{t=1}^m 
\frac{\varepsilon_t}{r_{t_k}}
 \left\{
 -
\mathrm{Li}_{2;d_{k_t},z_{k_t}[t]}(-y_{k_t}[t]^{\varepsilon_t})
-
\frac{1}{2}
 \log y_i[t]^{\varepsilon_t}
 \cdot
  \log 
P_{d_{k_t},z_{k_t}[t]}(y_{k_t}[t]^{\varepsilon_t})
 \right\},
\end{align}
which exactly yields  the {\em Rogers dilogarithms of higher degree} \eqref{eq:Rogers7}.
On the other hand, this term is 0 from the beginning.
Therefore, we have the classical dilogarithm identity of higher degree \eqref{eq:gid6}
with complex (almost real) argument.
Taking $\lambda\rightarrow 1$, we recover the identity  \eqref{eq:gid6} with real argument.

\bibliography{../../biblist/biblist.bib}
\end{document}